\documentclass[11pt,a4paper]{article}
\usepackage[utf8]{inputenc}
\usepackage[T1]{fontenc}
\usepackage{geometry}
\usepackage{xcolor}
\usepackage{graphicx}
\usepackage{booktabs} 

\usepackage{amsmath}
\usepackage{amssymb}
\usepackage{mathtools}
\usepackage{amsthm}
\usepackage{graphicx}
\usepackage{indentfirst}
\usepackage{standalone} 
\usepackage{tablefootnote}
\usepackage{booktabs} 
\usepackage{multirow}
\usepackage{tabularx}
\usepackage{empheq}
\usepackage{algorithm}  
\usepackage{algpseudocode}
\usepackage[normalem]{ulem}
\useunder{\uline}{\ul}{}

\usepackage{enumitem}
\usepackage{pageslts}
\usepackage{bm}
\usepackage{authblk}

\usepackage{hyperref}
\hypersetup{
    colorlinks=true,
    linkcolor=blue,
    filecolor=magenta,      
    urlcolor=cyan,
    citecolor=blue,
}

\newenvironment{keywords}{
    \vspace{0.3cm} 
    \noindent\textbf{Keywords:} 
}{
    \vspace{0.3cm} 
}
\newcommand{\sep}{, }

\newtheorem{definition}{Definition} 
\newtheorem{theorem}{Theorem}
\newtheorem{remark}{Remark}
\newtheorem{proposition}{Proposition}
\newtheorem{lemma}{Lemma}

\title{iFCTN: an intra-block Fully-Connected Tensor Network Decomposition for Tensor Completion}
\author[1]{Ziyi Gan}
\author[1]{Chunfeng Cui\thanks{Corresponding author.}}
\affil[1]{School of Mathematical Sciences, Beihang University, Beijing, China \\
\vspace{0.1cm}
\texttt{ganziyi@buaa.edu.cn}, \quad \texttt{chunfengcui@buaa.edu.cn}}

\date{} 
\begin{document}
\pagenumbering{arabic}
\maketitle

\begin{abstract}
    The fully-connected tensor network (FCTN) decomposition has recently exhibited strong modeling capabilities by connecting every pair of tensor factors, thereby capturing rich cross-mode correlations. However, this advantage comes with an inherent limitation: updating the factors typically requires reconstructing auxiliary sub-networks, which entails extensive and cumbersome (un)folding. In this study, we propose the intra-block FCTN (iFCTN) decomposition, a novel (un)folding-free variant of FCTN decomposition designed to enhance computational efficiency. We parameterize each FCTN factor through Khatri-Rao products, which significantly reduces the complexity of reconstructing intermediate sub-networks and yields subproblems with well-structured coefficient matrices. Furthermore, we deploy the proposed iFCTN decomposition on the representative task of tensor completion and design an efficient proximal alternating minimization algorithm. Theoretically, we establish its global convergence to a critical point. Extensive experiments demonstrate that iFCTN outperforms state-of-the-art methods with a lower computational overhead.
\end{abstract}


\section{Introduction}

    With rapid advances in science and technology, high-dimensional datasets (e.g., multispectral images \cite{MSI1,MSI2,MSI3,MSI4} and traffic-flow measurements \cite{traffic1,traffic2,traffic3}) have become pervasive in practice and are naturally represented as tensors. A tensor extends the concept of a matrix to higher orders, with entries indexed by multiple (continuous or discrete) variables. Tensor decompositions are powerful tools for modeling and processing such high-order data, and their applications are ubiquitous in machine learning \cite{ml3,ml2,tensorwang2024tensorized,ml1}, signal processing \cite{sg1,sg2}, and much more \cite{bai,muchmore2,muchmore1}. Conceptually akin to matrix factorization, tensor decomposition is the art of disassembling a multi-dimensional array into smaller factors 
    \cite{TensorDecompositions}. Over the past decades, several prominent models have been established, including Tucker decomposition \cite{tuckerHitchcock1927, tucker1966some,tuckerzhang2025noisy}, CANDECOMP/PARAFAC (CP) decomposition \cite{kolda2009tensor,cpkiers2000towards}, tensor singular value decomposition (t-SVD) \cite{kilmer2011factorization,kilmer2013third}, and tensor network (TN) decompositions\cite{TNs1,TNs2,TNs3}.

    TN decomposition factorizes an $N$-th order tensor into a set of low-order constituent tensors (referred to as cores), where tensor contractions characterize the algebraic interactions among these factors. This paradigm inherently facilitates exceptional data compression and high computational efficiency \cite{TNSbetter, TNSbetter2, TNSbetter3}. Among existing frameworks, the Tensor Train (TT) \cite{TT} and Tensor Ring (TR) \cite{TR} decompositions are characterized by chain and ring topologies, respectively. As a foundational milestone, TT decomposition represents an $N$-th order tensor as a sequence of multilinear products across a set of cores. Specifically, the terminal cores are constrained to matrices, while the intermediate $N-2$ cores are third-order tensors (as illustrated in Fig.~\ref{fig:TNs}). Subsequently, TR decomposition generalizes the TT structure by substituting the boundary matrices with third-order cores and imposing a cyclic contraction, thereby establishing a closed-loop topology. Moreover, recent studies \cite{qiu2024balanced} bridge tensor networks and traditional rank minimization by revealing intrinsic connections between TN ranks (e.g., TR rank) and balanced unfolding induced tensor nuclear norms. 

    \begin{figure}[t]
        \centering
        \includegraphics[width=0.9\textwidth]{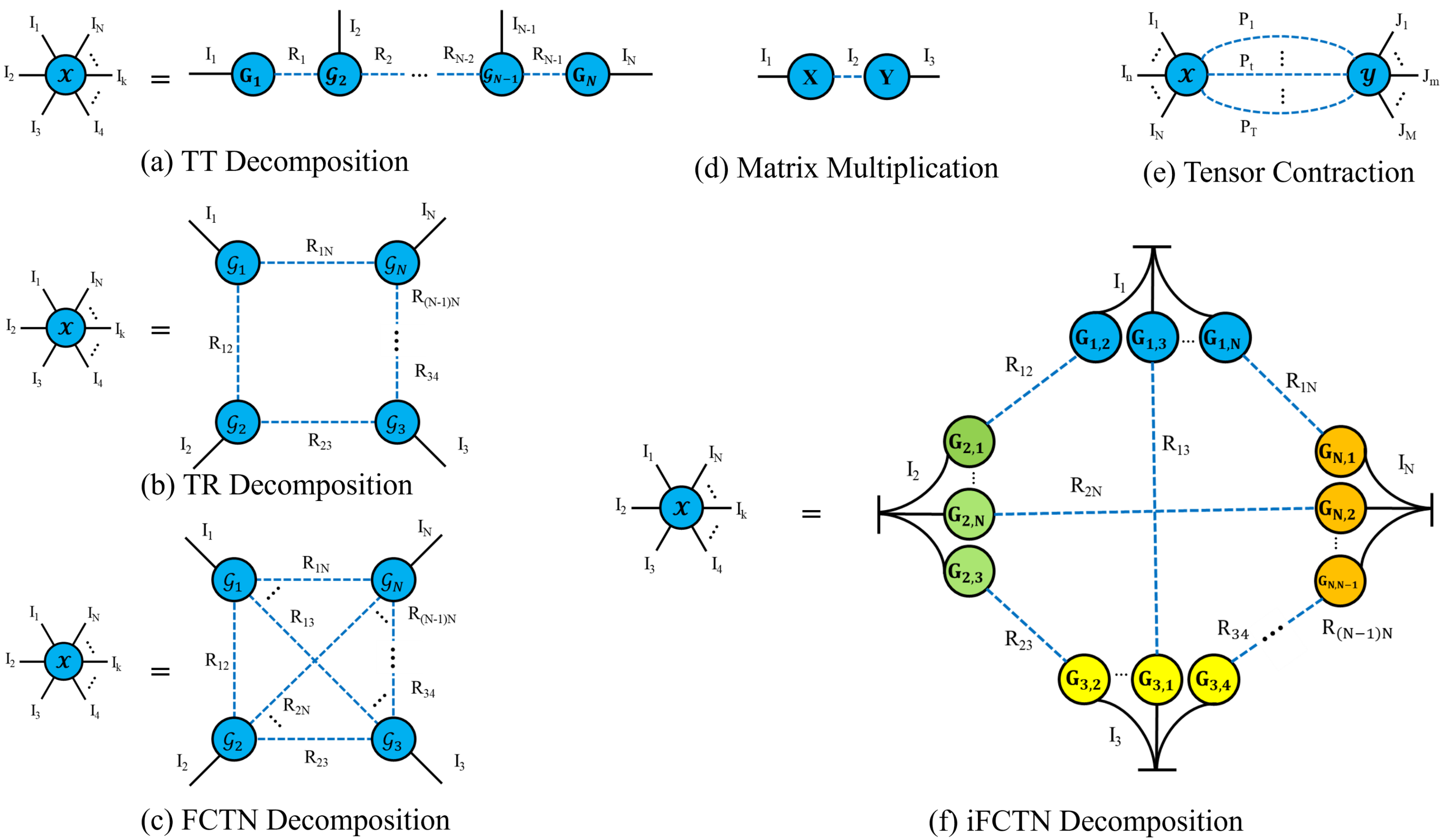}
        \caption{A graphical representation of TT, TR, FCTN and the proposed iFCTN decompositions. Here, the \textcolor{cyan}{blue dotted lines} denote matrix multiplications or tenor contractions, and the black horizontal line and curve resemble a ``cherry stem'' denote KR product. In contrast to the TT, TR, and FCTN, which employ tensor factors, iFCTN represents each tensor factor as a ``cherry chain'', whose components (``cherries'') are matrices.}
        \label{fig:TNs}
    \end{figure}
    
    Introduced in 2021, the fully connected tensor network (FCTN) decomposition \cite{FCTN} represents a significant advancement over TT and TR. Unlike its predecessors, which only establish connections between adjacent factors, FCTN decomposes an $N$-th order tensor into $N$ $N$-th order tensors, establishing a connection between \textit{every} pair of factors. This fully-connected topology grants FCTN superior capability to characterize the intrinsic correlations between any two modes of the tensor, making it highly effective for tensor completion (TC) \cite{FCTNrandom, FCTNROBUST, FCTNLRS, FCTNstrural}. Beyond data recovery, its potent representational ability has recently been extended to supervised machine learning paradigms, as evidenced by the development of the support fully-connected tensor networks machine and its associated safe screening rules for efficient classification \cite{wang2026support}. Subsequently, several variants of FCTN were also proposed, including SVDinsFCTN \cite{SVDinsFCTN}, N-FCTN \cite{NFCTN}, revealFCTN \cite{revealFCTN}, T-FCTN \cite{TFCTN}, NL-FCTN, and nonlocal patch-based FCTN \cite{FCTNNONLOCAL}.
    
    Although FCTN offers full connectivity, the high-order dimensionality of its core tensors inherently incurs prohibitive computational complexity and significant parameter redundancy \cite{FCTNSGD}.
    Specifically, updating FCTN factors typically requires reconstructing auxiliary sub-networks and repeatedly (un)folding intermediate tensors. Specifically, for an $N$th-order tensor, updating any single factor necessitates forming a dedicated sub-network from the remaining $N-1$ factors. Consequently, an $N$-way tensor entails $N$ such constructions per iteration, and these sub-networks must be rebuilt after every update. This reconstruction, with high computational cost and substantial memory overhead, slows practical training—particularly for large-scale, high-order tensors common in computer vision. 
     Addressing this computational burden while preserving FCTN's expressive power is a critical area of research \cite{FCTNDEEP}.

    \begin{figure}[t]
        \centering
        \includegraphics[scale=0.2]{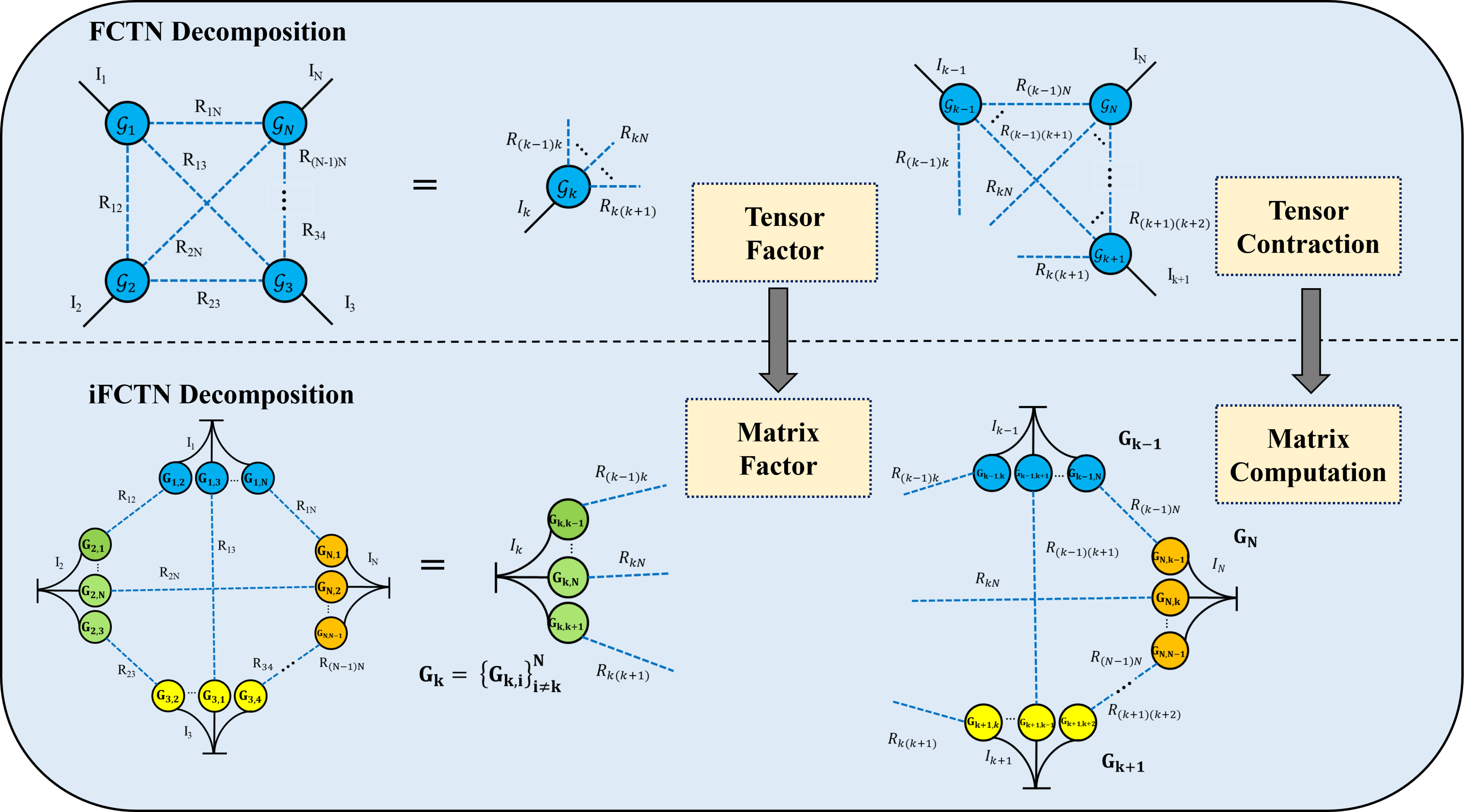}
        \caption{A comparative analysis of sub-network: FCTN versus the proposed iFCTN decomposition.}
        \label{fig:Subnet}
        \end{figure}

     To address these challenges, we revisit FCTN from an intra-block perspective and parameterize each tensor factor via Khatri–Rao (KR) products. This simple yet pivotal design preserves cross-mode expressivity while rendering the underlying linear algebra unfolding-free and amenable to efficient solvers. The main \textbf{\textit{contributions}} of this paper are summarized as follows: 
     
    \begin{itemize}
     \item We first propose a novel intra-block FCTN (iFCTN) framework, an (un)folding-free variant of FCTN that relies solely on matrix factors while capturing rich cross-mode correlations among the tensor factors (Section \ref{sec:iFCTN}). By decomposing the tensor factors into matrix factors (i.e., cherry tensor and cherry products) via the KR product, iFCTN admits an explicit expression for the sub-network (i.e. cherry sub-network).

     \item We then apply iFCTN to the TC problem (iFCTN-TC), directly leveraging the cherry-based parameterization for scalable optimization (Section \ref{sec:iFCTN-TC}). Our complexity analysis reveals that iFCTN effectively reduces both storage and computational complexities.
      Moreover, we establish convergence guarantees for the proposed proximal alternating minimization (PAM) solver.

    \item We assess the performance of iFCTN-TC on five benchmark datasets under eight random and fiber-wise missing scenarios. It consistently attains superior accuracy with lower computational cost, thereby validating its effectiveness and robustness (Section \ref{sec:Experiments}).
    \end{itemize}

\section{Preliminaries}
\subsection{Notations}
    In this paper, we denote vectors, matrices and tensors as $\bf{x}$, $\bf{X}$, $\mathcal{X}$, respectively. Given an $N$th-order tensor $\mathcal{X} \in \mathbb{R}^{I_{1} \times I_{2} \times \cdots \times I_{N}}$,  is denoted as $\mathcal{X}(i_1, i_2, \dots, i_N)$, where $i_n \in [I_n]$, $n \in [N]$, and $[N]$ represents the set $\{1, 2, \dots, N\}$. For index $(i_1, i_2, \dots, i_N)$, the notation $\overline{i_1 i_2 \cdots i_N} \stackrel{\text{def}}{=} 1 + \sum_{j=1}^{N}(i_j-1) \prod_{n=1}^{j-1} I_n$ is useful for unfolding matrices of a tensor. Its mode-$n$ unfolding is $\mathcal{X}_{(n)} \in \mathbb{R}^{I_{n} \times \prod_{k = 1,k\neq n}^{N}I_k}$. Here, $n\in [N]$ and  $[N] = \{1,\dots,N\}$.  
    Furthermore, given two matrices  $\mathbf{A} \in \mathbb{R}^{m \times r}$ and $\mathbf{B} \in \mathbb{R}^{n \times r}$, their Khatri-Rao product is   $\mathbf{C} = \mathbf{A} \odot \mathbf{B} \in \mathbb{R}^{mn \times r}$, where $c_{(i-1)n+j,l} = a_{il} b_{jl}$ for any $i\in[m]$, $j\in[n]$, and $l\in [r]$.    
    
\subsection{FCTN decomposition}   
    Given    $\mathcal{X} \in \mathbb{R}^{I_1 \times I_2 \times \cdots \times I_N}$, FCTN decomposes it into $N$ tensor factors $\mathcal{G}_k \in \mathbb{R}^{(\prod_{j<k} R_{j,k}) \times I_k \times (\prod_{j>k} R_{k,j})}$ for any $k\in[N]$. Specifically, the entry-wise form of the FCTN decomposition is given by:
        \begin{equation}\label{FCTN}
        \begin{aligned}
            & \mathcal{X}\left(i_1, i_2, \dots, i_N\right) =  \sum_{r_{1,2}=1}^{R_{1,2}} \sum_{r_{1,3}=1}^{R_{1,3}} \dots \sum_{r_{1, N}=1}^{R_{1, N}} \sum_{r_{2,3}=1}^{R_{2,3}} \dots \sum_{r_{2, N}=1}^{R_{2, N}} \dots \sum_{r_{N-1, N}=1}^{R_{N-1, N}} \left\{\mathcal{G}_1\left(i_1, r_{1,2}, r_{1,3}, \dots, r_{1, N}\right) \cdot \right. \\
            & \left.\mathcal{G}_2\left(r_{1,2}, i_2, r_{2,3}, \dots, r_{2, N}\right) \cdots \mathcal{G}_k\left(r_{1, k}, r_{2, k}, \dots, r_{k-1, k}, i_k, r_{k, k+1}, \dots, r_{k, N}\right) \cdots \right. \\
            & \left. \mathcal{G}_N\left(r_{1, N}, r_{2, N}, \dots, r_{N-1, N}, i_N\right)\right\}.
        \end{aligned}
        \end{equation}
     We denote the FCTN decomposition by \begin{equation*}
         \mathcal{X}= \operatorname{FCTN}\left(\left\{\mathcal{G}_k\right\}_{k=1}^N\right)=\operatorname{FCTN}\left(\mathcal{G}_1, \mathcal{G}_2, \dots, \mathcal{G}_N\right) 
     \end{equation*} 
     and call the vector collected by $\{R_{k_1, k_2}\}$ as the FCTN-ranks. 

    Recently, FCTN and its variants \cite{FCTNrandom, FCTNSGD} demonstrate high efficiency and stability in many applications. However, computing the FCTN decomposition,  demands either numerous summation loops or extensive large-scale tensor folding/unfolding to leverage matrix operations. To address these limitations, this paper introduces a relaxation technique based on the KR product.

\section{iFCTN Decomposition}\label{sec:iFCTN}
    Our proposed iFCTN is derived from FCTN, as depicted in Fig.~\ref{fig:TNs}~(d). The pivotal alteration lies in parameterizing each tensor factor as a sequence of matrices, thereby achieving a matrix-based realization of  FCTN. This reformulation enables all updates to be performed via matrix computations without tensor contractions. Furthermore, by leveraging  the KR product, we establish several structural results that underpin the proposed iFCTN.  
   
    We begin with the following definition, which decomposes a tensor into small matrices factors.
    
    \begin{definition}[\textbf{Mode-$k$ Cherry Tensor}]\label{def:cherrytensor}
    A tensor $\mathcal{G}\in \mathbb{R}^{R_{1:k-1}\times  I_{k}\times R_{k+1:N}}$ is referred to as a mode-$k$ KR-product tensor, or simply a mode-$k$ cherry tensor, if its $k$-th mode unfolding admits a KR factorization as follows,
    \begin{equation} \label{equ:KR-tensor}
    \left(\mathcal{G}_{\left(k\right)}\right)^{\top} = \textbf{M}_{N} \odot \cdots \odot \textbf{M}_{k+1} \odot \textbf{M}_{k-1} \cdots \odot \textbf{M}_{1},
    \end{equation} 
    where $\textbf{M}_{i}\in\mathbb{R}^{R_{i}\times I_{k}}$ are cherry factors of $\mathcal{G}$, denoted by $\mathcal{CT}$. For convenience, we also denote 
    \begin{equation*}
        \mathcal G = \mathcal{CT}_k\left(\{\textbf{M}_i\}_{i=1,i\neq k}^N\right).
    \end{equation*}
    \end{definition}
    
    We term $\mathcal G$ in \eqref{equ:KR-tensor} a cherry tensor, as its structure, depicted in Fig.~\ref{fig:TNs}, visually resembles ``a string of cherries'' hanging from their ``stems'': the KR product forms the ``stems'' that connect the smaller matrices (``cherries'').  This representation simultaneously preserves an explicit dependence on the underlying tensor structure. 

    
    In practice, there is no need to construct the cherry tensor \eqref{equ:KR-tensor}. Instead, all storage and computations are performed directly on the cherry factors $\{\mathbf{G}_{k,i}\}$.  
    For brevity, we hereafter denote the set of cherry factors at the $k$-th mode (``$k$-th string of cherries'') by
    \begin{equation}\label{equ:cherry_factor_set}
    \textbf{G}_{k} =\left\{\textbf{G}_{k,i}\right\}_{i\neq k}^{N}, \quad \textbf{G}_{k,i} \in \mathbb{R}^{ R_{k,i} \times I_{k}}.
    \end{equation}
    With Definition~\ref{def:cherrytensor}, the (un)folding operations are explicitly formulated in matrix form, which streamlines the construction of the tensor network. This entire process formally defines the iFCTN decomposition, as presented below.

    \begin{definition}[\textbf{iFCTN Decomposition}]
        The iFCTN decomposition decomposes  an $N$th-order tensor $\mathcal{X} \in \mathbb{R}^{I_1 \times I_2 \times \cdots \times I_N}$ into a set of cherry tensors in
        Definition~\ref{def:cherrytensor},
        which are composed of the corresponding cherry factors $\textbf{G}_{k}$ as in \eqref{equ:cherry_factor_set}. Suppose the iFCTN-rank matrix $\textbf{R}=(R_{i,j})\in\mathbb R^{N\times N}$  is symmetric.  Then, the element-wise form of the iFCTN decomposition can be expressed as follows,
        \begin{equation}\label{eq:iFCTN}
        \begin{aligned}
                 \mathcal{X}(i_{1},i_{2},\dots,i_{N})  = & \sum_{r_{1,2}=1}^{R_{1,2}} \dots \sum_{r_{1, N}=1}^{R_{1, N}} \sum_{r_{2,3}=1}^{R_{2,3}} \dots \sum_{r_{2, N}=1}^{R_{2, N}} \cdots \sum_{r_{N-1, N}=1}^{R_{N-1, N}} \left\{  \textbf{G}_{1,2} (r_{1,2},i_{1}) \cdot
                \textbf{G}_{1,3} (r_{1,3},i_{1}) \dots \right.\\
                & \left. \textbf{G}_{1,N}(r_{1,N},i_{1}) \cdot \textbf{G}_{2,1}(r_{1,2},i_{2}) \cdot \textbf{G}_{2,3}(r_{2,3},i_{2})\cdots \textbf{G}_{2,N}(r_{2,N},i_{2})\cdots \right. \\
                & \left. \textbf{G}_{N,1}(r_{1,N},i_{N}) \cdot \textbf{G}_{N,2}(r_{2,N},i_{N})\cdots \textbf{G}_{N,N-1}(r_{N-1,N},i_{N})\right\}.
        \end{aligned}
        \end{equation}
         We denote the iFCTN decomposition by 
         \begin{equation*}
             \mathcal{X}= \operatorname{iFCTN}\left(\left\{\textbf{G}_{k}\right\}_{k=1}^N\right)=\operatorname{iFCTN}\left(\textbf{G}_{1},\dots,\textbf{G}_{N}\right)
         \end{equation*}
         and call the vector collected by $\{R_{k_1, k_2}\}(k_1<k_2)$ as the iFCTN-ranks. 
    \end{definition}

    \begin{remark}
    When $N=3$, the iFCTN decomposition operates on a third-order tensor $\mathcal{X} \in \mathbb{R}^{I_1 \times I_2 \times I_3}$. The symmetric iFCTN-rank matrix $\textbf{R}$ involves three pairwise ranks: $R_{1,2}, R_{1,3},$ and $R_{2,3}$. The element-wise representation \eqref{eq:iFCTN} simplifies to:
    \begin{equation*}
    \begin{aligned}
        \mathcal{X}(i_1, i_2, i_3) = \sum_{r_{1,2}=1}^{R_{1,2}} \sum_{r_{1,3}=1}^{R_{1,3}} \sum_{r_{2,3}=1}^{R_{2,3}} & \left\{ \textbf{G}_{1,2}(r_{1,2}, i_1) \cdot \textbf{G}_{1,3}(r_{1,3}, i_1) \cdot \textbf{G}_{2,1}(r_{1,2}, i_2) \cdot \right.\\
        & \left. \textbf{G}_{2,3}(r_{2,3}, i_2) \cdot \textbf{G}_{3,1}(r_{1,3}, i_3) \cdot \textbf{G}_{3,2}(r_{2,3}, i_3) \right\}.    
    \end{aligned}
    \end{equation*}
    In this case, the tensor is formed by three cherry factors, denoted as $\mathcal{X} = \operatorname{iFCTN}(\textbf{G}_1, \textbf{G}_2, \textbf{G}_3)$. Each cherry factor $\textbf{G}_k$ captures the interactions with the other two modes (e.g., for $k=1$, $\textbf{G}_{1,2}$ and $\textbf{G}_{1,3}$ couple mode-1 with mode-2 and mode-3, respectively). The iFCTN-ranks are collected as the vector $[R_{1,2}, R_{1,3}, R_{2,3}]$.
    \end{remark}

    \begin{remark}
    In TN decompositions, such as TT, TR, FCTN, TW, every factor is a tensor of similar or lower order. Computing the $k$-th factor in a TN typically entails ``packing'' step that reformulates the remaining components into an effective sub-network. As illustrated in Fig.~\ref{fig:Subnet}, the FCTN rebuilds an auxiliary sub-network at every iteration with repeated contractions, permutations, and folding/unfolding steps between the remaining tensors. 
    
    By contrast, all factors in iFCTN are matrices (``cherries'') and ``packing'' the remaining cherries into a sub-network can be explicitly represented with these matrices.
    \end{remark}
    
    Building upon this matrix-centric representation, we now introduce the \textit{Cherry Product} to formalize the algebraic mechanism for 'packing' and interconnecting these sub-networks into higher-order structures.

    \begin{definition}[\textbf{Cherry Product}]
        Consider two cherry tensors $\mathcal{A}=\mathcal{CT}_k(\{\mathbf{A}_i\}_{i\neq k}) \in \mathbb{R}^{P_{1:k-1} \times I_{k} \times P_{k+1:N}}$ and   $\mathcal{B} =\mathcal{CT}_t(\{\mathbf{B}_j\}_{j\neq t})\in \mathbb{R}^{Q_{1:t-1} \times J_{t} \times Q_{t+1:M}}$. Let $\alpha \in [N] \setminus k$, $\beta \in [M]\setminus t$, $\textbf{n} = [N] \setminus \alpha$, and $\textbf{m} = [M]\setminus \beta$. Suppose that    $P_\alpha = Q_\beta$.  
        Then their cherry product yields an $(N+M-2)$ th order tensor as follows, 
        \begin{equation}
            \begin{aligned}
                & \mathcal{C}(i_{n_1},\dots,i_{n_{N-1}},j_{m_1},\dots,j_{m_{M-1}}) =  \prod_{i=1}^{N-1}\textbf{A}_{n_i}(i_{n_i},i_{n_{\tilde k}}) \prod_{j=1}^{M-1}\textbf{B}_{m_j}(j_{m_i},j_{m_{\tilde t}}) (\textbf{A}_{\alpha}^{\top}\textbf{B}_{\beta})_{i_{n_{\tilde k}},j_{m_{\tilde t}}}.
            \end{aligned}
        \end{equation}
    Here, $\tilde{k} = k$ if $\alpha > k$ and $k-1$ otherwise; similarly, $\tilde{t} = t$ if $\beta > t$ and $t-1$ otherwise.  
    \end{definition}

    \begin{remark}
        For simplicity, let   $M = N = 3$, and consider a mode-$2$ cherry tensor $\mathcal{A} \in \mathbb{R}^{P \times I \times R}$ and a  mode-$3$ cherry tensor $\mathcal{B} \in \mathbb{R}^{Q \times R \times J}$. Then their Cherry product  yields a tensor of size $\mathbb R^{P\times I\times Q\times J}$ written as follows,
        $$
        \mathcal{C}(p,i,q,j) = \textbf{A}_{1}(p,i) \cdot \textbf{B}_{1}(q,j) \cdot (\textbf{A}_{3}^{\top}\textbf{B}_{2})_{i,j},
        $$
        where  $\textbf{A}_{1} \in \mathbb{R}^{P \times I}$,  $\textbf{A}_{3} \in \mathbb{R}^{R \times I}$, $\textbf{B}_{1} \in \mathbb{R}^{Q\times J}$, and $\textbf{B}_{2} \in \mathbb{R}^{R \times J}$. 
        Or equivalently, we can matricize $\mathcal C$ into
        $\textbf{C} \in \mathbb{R}^{IJ \times R_1R_3}$ 
        {as follows,}
        $$
        \textbf{C} = \operatorname{diag}(\operatorname{vec}(\textbf{A}_3^{\top} \textbf{B}_{2})) \cdot (\textbf{B}_1 \otimes \textbf{A}_1)^{\top}.
        $$
    \end{remark}

    The cherry product operation can be visually analogized to a tile-matching game: when two ``cherry factors'' are connected, they undergo a matrix multiplication, and the corresponding rank information is absorbed. Through a sequence of such cherry products, we can ultimately construct a tensor that has the same dimension as the original target tensor. 

     \begin{proposition}[\textbf{Mode-$k$   Cherry Sub-network}]\label{propos:cherrysubnet}
        Suppose that $\mathcal{X} = \operatorname{iFCTN} \left(\left\{\textbf{G}_{t}\right\}_{t=1}^{N}\right)$. 
        Then, the cherry sub-network $\textbf{Z}_{k} \in \mathbb R^{\prod_{i \in [N] \setminus k} I_{i} \times \prod_{i \in [N] \setminus k} R_{k,i}}$  comprises two components: (i) the cherry factors $\{\mathbf{G}_{t,i}\}_{i\neq t}$ for all  $t\in[N]\setminus k$; (ii) the   cherry factors $\textbf{G}_{k}$. 
        Specifically, 
        \begin{equation}\label{equ:Zk}
            \textbf{Z}_{k} =  \mathrm{diag}\left(\mathrm{vec}\left(\mathcal{S}_{k}\right)
             \right) \times  \mathop{\otimes} \limits_{i \in \left\{N,\dots,1\right\} \setminus k} \textbf{G}_{i,k}^{\top},
        \end{equation}
            where $\mathcal{S}_{k}$ is an $(N-1)$ th-order tensor shown as follows,
        \begin{align*}
             \mathcal{S}_{k} = & \mathop{\circ} \limits_{i<j, i, j \in [N] \setminus k} \left\{\operatorname{reshape}\left(\mathbf{G}_{i,j}^{\top}\mathbf{G}_{j,i},[1,\dots,I_{i},\dots,I_{j},1,\dots,1]\right)\right\}.  
     \end{align*}
     Moreover, the mode-$k$ unfolding of $\mathcal{X}$ can be expressed by:  
        \begin{equation}
            \left(\mathcal{X}_{\left(k\right)}\right)^{\top} = \textbf{Z}_{k} \left(\mathcal{G}_{(k)}\right)^{\top},
        \end{equation}
        where $\mathcal{G}_{(k)}=\mathcal{CT}_k\left(\{{\bf G}_i\}_{i=1,i\neq k}^N\right).$
     \end{proposition}

     Proposition \ref{propos:cherrysubnet} demonstrates the profound algorithmic efficiency of the iFCTN framework. When solving the optimization sub-problems to update individual factors, evaluating the corresponding sub-network is a computational bottleneck. Unlike the original FCTN, which relies heavily on demanding high-order tensor contractions, iFCTN elegantly parameterizes its factors as matrices. Rather than complicating the model, this matrix-based formulation allows the cherry sub-network to be explicitly and directly constructed via standard matrix operations. Consequently, by casting the factor-updating sub-problems into a closed matrix form, the proposed method avoids complex high-order tensor manipulations, thereby providing a more straightforward and computationally friendly approach for the solving process.


    \begin{proposition}[\textbf{Uniqueness}]\label{propos:Uniqueness}
        Let $\mathcal{X}$ be an $N$-th order tensor admitting an iFCTN decomposition with rank $\mathbf{R}$. The cherry factors $\{M_{k,j}\}_{j \neq k}$ in iFCTN are essentially unique (up to permutation and scaling), provided that Kruskal's condition is satisfied. Specifically, for any mode $k$, if the sum of k-ranks of the cherry factors satisfies:
        $$\sum_{j \neq k} k_{\text{rank}}(M_{k,j}) \ge  I_k + N-1,$$ 
        then the factors $\{M_{k,j}\}$ are unique up to a common permutation matrix $\mathbf{\Pi}$ and scaling diagonal matrices $\{\mathbf{\Lambda}_j\}$ such that $\prod \mathbf{\Lambda}_j = \mathbf{I}$.
    \end{proposition}

     Proposition \ref{propos:Uniqueness} establishes a crucial theoretical property of the iFCTN decomposition: the strict identifiability of its latent factors. While broad classes of Tensor Network decompositions (including TT, TR, and the original FCTN) are typically hindered by gauge or rotational ambiguities, iFCTN successfully overcomes this limitation. It achieves essential uniqueness under mild Kruskal-type conditions, effectively mirroring the highly desirable uniqueness properties of the CP decomposition within a sophisticated TN framework. Consequently, the extracted cherry factors are guaranteed to be physically meaningful and robustly recoverable, facilitating interpretable feature extraction without relying on complex artificial regularizations.

         \begin{proposition}[\textbf{Transposition Invariance}]\label{propos:Transposition}
        Suppose that an $N$th-order tensor $\mathcal{X} \in \mathbb{R}^{I_1 \times \dots \times I_N}$ has the following iFCTN decomposition parameterized by the set of cherry factors $\mathbf{M} = \{\mathbf{M}_{k,j} \mid 1 \le k,j \le N, k \neq j\}$:
        $$\mathcal{X} = \text{iFCTN}(\{\mathbf{M}_{k,j}\}_{k \neq j}).$$
        Then, its vector $\mathbf{n}$-based generalized tensor transposition $\vec{\mathcal{X}}^{\mathbf{n}}$ can be expressed as:
        $$\vec{\mathcal{X}}^{\mathbf{n}} = \text{iFCTN}(\{\mathbf{M}^{\mathbf{n}}_{k,j}\}_{k \neq j}),$$
        where $\mathbf{n}=(n_{1},n_{2},\cdot\cdot\cdot,n_{N})$ is a reordering of the vector $(1,2,...,N)$, and the new factors are given by the mapping $\mathbf{M}^{\mathbf{n}}_{k,j} = \mathbf{M}_{n_k, n_j}$.
    \end{proposition}

    Proposition \ref{propos:Transposition} demonstrates that the iFCTN decomposition successfully inherits the generalized transposition invariance property from the original FCTN framework. This represents a fundamental advantage over standard sequential tensor networks such as the TT and TR decompositions. Specifically, the iFCTN formulation remains inherently invariant under any arbitrary permutation of the target tensor's modes. In contrast, the invariance of the TR decomposition is strictly limited to circular shifts or reverse permutations of the modes, while the TT decomposition is even more restricted, maintaining invariance only under reverse permutations.

\section{iFCTN Decomposition-Based TC Method}\label{sec:iFCTN-TC}

  In the following, the iFCTN decomposition method is employed to address the TC problem, with the aim of validating its rationality and superiority. 

\subsection{Model}

    Given an observation set $\Omega$ and an incomplete observation $\mathcal{O} \in \mathbb{R}^{I_{1} \times I_{2} \times \cdots \times I_{N}}$, we aim to recover a target tensor $\mathcal{X}$ that is low-rank within the iFCTN framework. Then, the proposed iFCTN-TC can be formulated as follows,
    \begin{equation}\label{eq:iFCTN-TC}
        \min _{\mathcal{X}, \left\{\textbf{G}_{k}\right\}_{k=1}^{N}} \frac{1}{2}\left\|\mathcal{X}-\operatorname{iFCTN}\left(\textbf{G}_{1},\dots, \textbf{G}_{N}\right)\right\|_{F}^{2}+\iota (\mathcal{X}). 
    \end{equation}
    Here,   $\iota (\mathcal{X})$  is an indicator function that enforces $\mathcal{X}(\mathbf i) = \mathcal{O}(\mathbf i)$ for any $\mathbf i \in \Omega$ defined by 
    \begin{equation}
        \iota (\mathcal{X}):=\begin{cases} 0, &\text { if }  \mathcal{P}_{\Omega}(\mathcal{X})=\mathcal{P}_{\Omega}(\mathcal{O}), \\ 
        \infty,&  \text { otherwise},\end{cases}
    \end{equation}
    where $\mathcal{P}_{\Omega} (\cdot)$ is a projection function. 
    For convenience, we also denote $\textbf{G} := \{\textbf{G}_{k}\}_{k=1}^{N}$, and define:
    \begin{align*}
        f(\textbf{G}, \mathcal{X}) &= \frac{1}{2}\left\|\mathcal{X}-\operatorname{iFCTN}(\textbf{G}_{1},\dots, \textbf{G}_{N})\right\|_{F}^{2}, \\
        F(\textbf{G}, \mathcal{X}) &= \frac{1}{2}\left\|\mathcal{X}-\operatorname{iFCTN}(\textbf{G}_{1},\dots, \textbf{G}_{N})\right\|_{F}^{2}+\iota (\mathcal{X}).
    \end{align*}
    
    Although the objective function $F(\textbf{G}, \mathcal{X})$ in \eqref{eq:iFCTN-TC} is nonconvex, it is convex with respect to each block when the remaining blocks are fixed. This motivates us to employ the PAM scheme to iteratively update the variables as follows. 

    \subsection{Algorithm}

    At the $s$-th iteration, we begin by sequentially updating the cherry factors for all $k \in [N]$ and $i \neq k$ as follows:
    \begin{align*}
        \textbf{G}_{k,i}^{(s+1)} = &\underset{\textbf{G}_{k,i}}{\operatorname{argmin}}\left\{f\left(\mathbf{G}_{<k,i}^{(s+1)},\mathbf{G}_{k,i},\mathbf{G}_{>k,i}^{(s)},\mathcal{X}^{(s)}\right)\right.  \left.+ \frac{\rho}{2}\left\|\textbf{G}_{k,i}-\textbf{G}_{k,i}^{(s)}\right\|_{F}^{2}\right\}.
    \end{align*}
   Then, we update target tensor $\mathcal X$ by:    
    \begin{equation}\label{Xupdate}
        \mathcal{X}^{(s+1)} = \underset{\mathcal{X}}{\operatorname{argmin}}\left\{F\left(\textbf{G}^{(s+1)}, \mathcal X\right)+\frac{\rho}{2}\left\|\mathcal{X}-\mathcal{X}^{(s)}\right\|_{F}^{2}\right\}.  
    \end{equation}
    Here, $\rho>0$ denotes the parameter controlling the proximal term, $\mathbf{G}_{<,k,i}^{(s+1)}
    := \big\{\mathbf{G}_{1:k-1}^{(s+1)},\,\mathbf{G}_{k,1:i-1}^{(s+1)}\big\},$ and $    \mathbf{G}_{>,k,i}^{(s)}
    := \big\{\mathbf{G}_{k,i+1:N}^{(s)},\,\mathbf{G}_{k+1:N}^{(s)}\big\}.$ 

\textit{    1) Update $\textbf{G}_{k,i}$ : }According to Proposition \ref{propos:cherrysubnet}, 
the $\mathbf{G}_{k,i}$ subproblem could be rewritten as follows,
    \begin{align}
         & \textbf{G}_{k,i}^{(s+1)} = \underset{\textbf{G}_{k,i}}{\operatorname{argmin}}\left\{ \frac{\rho}{2}\left\|\textbf{G}_{k,i}-\textbf{G}_{k,i}^{(s)}\right\|_{F}^{2} \right.\label{eq:subproblem} \left.+ \left\|\textbf{X}_{(k)}^{\top} -\textbf{Z}_{k}^{(s)}\left(\textbf{G}_{k,N}^{(s+1)}\odot\dots\odot\textbf{G}_{k,i}\odot\dots\odot\textbf{G}_{k,1}^{(s)}\right)\right\|_{F}^{2} \right\}, \nonumber
    \end{align} 
    where $\textbf{Z}_{k}^{(s)}$ is the $k$-th cherry subnetwork defined by \eqref{equ:Zk}. This is a linear least square over $\textbf{G}_{k,i}$, which admits a closed-form solution. Since each column of $\textbf{G}_{k,i}$ is separable in the subproblem, we therefore solve it column-wise to avoid constructing the large coefficient matrix explicitly. Specifically, for each $j\in[I_k]$, we could compute the $j$-th column of $\mathbf{G}_{k,i}$ as follows,
    \begin{equation}\label{eq:updateG}
        \begin{aligned}
        & \left(\textbf{G}_{k,i}^{(s+1)}\right)_{j} = \left(\left(\textbf{y}_{j}^{\top} \otimes \textbf{G}_{i,k}^{(s)} \right) \textbf{D}^{2} \left(\textbf{y}_{j} \otimes \textbf{G}_{i,k}^{(s)\top}\right)+ \rho \textbf{I} \right)^{-1}\mathbf{b}_{i,k,j},  
        \end{aligned}
    \end{equation}
    where $\textbf{D} = \operatorname{diag}(\operatorname{vec}(\mathcal{S_{k}}))$, $\textbf{y}_{j}$ is the $j$-th column of $\textbf{Y} = \mathop{{\bigodot}}\limits_{m \neq i} \left(\textbf{G}_{k,m}^{(s)\top}\textbf{G}_{m,k}^{(s)}\right)\in \mathbb{R}^{\left(\prod_{m\neq i,k}^{N} I_{m}\right)\times I_{k}}$, and $\mathbf{b}_{i,k,j}=\left(\textbf{y}_{j}^{\top} \otimes \textbf{G}_{i,k}^{(s)} \right) \textbf{D} \left(\textbf{X}_{k}^{\top}\right)_{j}+\rho \left(\textbf{G}_{k,i}^{(s)}\right)_{j}$. 
    

    \textit{2) Update $\mathcal{X}$:}  
    The loss function in \eqref{Xupdate} is completely separable across all elements of $\mathcal X$. Thus, for each index $\mathbf{i} \in \Omega$, the update is fixed as $\mathcal{X}^{(s+1)}(\mathbf{i}) = \mathcal{O}(\mathbf{i})$; for all other entries, it is updated by solving the corresponding least squares problem. This procedure yields the following closed-form solution:
    \begin{equation}\label{eq:updateX}
        \mathcal{X}^{(s+1)} = \mathcal{P}_{\Omega^{c}}\left(\frac{\operatorname{iFCTN}\left(\textbf{G}\right)+\rho \mathcal{X}^{(s)}}{1+\rho}\right)+\mathcal{P}_{\Omega}(\mathcal{O}). 
    \end{equation}
    Here, $\Omega^c$ denotes the complement of $\Omega$.

    The whole process of the PAM-based solver for iFCTN-TC is summarized in Algorithm \ref{alg:iFCTN-TC}. It is noteworthy that the column-wise update scheme formulated in Eq.~\ref{eq:updateG} exhibits a high degree of mathematical decoupling, making the proposed iFCTN intrinsically amenable to massive GPU acceleration. Specifically, the computation for the $j$-th column is strictly independent of any other columns, essentially decomposing the factor update into $I_k$ embarrassingly parallelizable linear sub-systems. Unlike traditional tensor network contractions that typically suffer from high-latency, non-coalesced memory accesses due to recursive un(folding) and permutation operations, Eq.~\ref{eq:updateG} exclusively relies on well-structured dense matrix operations. 

    \begin{algorithm}[h]
    \caption{A PAM-based solver for iFCTN-TC.}  
    \label{alg:iFCTN-TC}      
    \textbf{Input:}  The observed tensor $\mathcal{O}\in\mathbb{R}^{I_{1}\times I_{2}\times\cdots\times I_{N}}$, the sample set $\Omega$, the iFCTN-ranks $\textbf{R}\in\mathbb R^{N\times N}$, 
    $\rho$, $t_{\max}$, and $\epsilon$.
   
    \textbf{Initialization:} Set $s=0$, $\mathcal{X}^{(0)}=\mathcal{O}$, and  $\textbf{G}_{k,i}^{(0)}\in\mathbb{R}^{R_{k,i}\times I_{k}}$.
    \begin{algorithmic}[1]
        \While{not converged and $s<t_{\max}$}
           \For{$k\in[N]$, $i\in [N]\setminus k$, and $j\in[I_k]$,}
            \State Update $\left(\textbf{G}_{k,i}^{(s+1)}\right)_{j}$ via (\ref{eq:updateG}).
            \EndFor
            \State Update $\mathcal{X}^{(s+1)}$ via (\ref{eq:updateX}).
            \State Check the convergence condition:
            \begin{align}
                \left\|\mathcal{X}^{(s+1)}-\mathcal{X}^{(s)}\right\|_{F}/\left\|\mathcal{X}^{(s)}\right\|_{F} < \epsilon. \nonumber
            \end{align}
            \State $s \leftarrow s+1 $
        \EndWhile
    \end{algorithmic}
    \textbf{Output:} The recovered tensor $\mathcal{X}$
    \end{algorithm}


\begin{table*}[h]
    \centering
    \caption{Storage and computational complexity comparison between FCTN and iFCTN (ours) for 3rd and $N$-th order tensors. Here, we assume that $I=I_1=\dots =I_N$ and $R=R_{ij}$ for any $i,j\in[N]$.}
    \label{tab:complexity}
    \begin{tabular*}{\textwidth}{@{\extracolsep{\fill}}cccc}
        \toprule
        Order & Method  & Storage complexity & Computation complexity \\
        \midrule
        \multirow{2}{*}{3-D} & FCTN & $\mathcal{O}(IR^{2})$ & $\mathcal{O}(I^{3}R^{2})$\\
                             & iFCTN & $\mathcal{O}(IR)$ & $\mathcal{O}(I^3)$ \\
        \midrule
        \multirow{2}{*}{$N$-D} & FCTN & $\mathcal{O}\left(N IR^{N-1}\right)$ & $\mathcal{O}\left(N \sum_{k=2}^{N} I^{k} R^{k(N-k)+k-1}+\right.$ $\left.N I^{N-1} R^{2(N-1)}+N R^{3(N-1)}\right)$ \\
                               & iFCTN & $\mathcal{O}\left(N(N-1)IR\right)$ & $\mathcal{O}(N I^{N} + N^3 I^{N-1} + N^3 I^{2}R + N IR^{3})$ \\
        \bottomrule
    \end{tabular*}
\end{table*}

\subsection{Complexity analysis}

    In this subsection, we assume that
    $\mathcal{X}\in\mathbb{R}^{I\times I\times\cdots\times I}$ and $R_{k_1,k_2} = R_{k_2,k_1} = R$ for any $k_1,k_2\in[N]$.
    We present a comparison between iFCTN and FCTN \cite{FCTN} in  Table~\ref{tab:complexity}.

    \textbf{Storage complexity.} In iFCTN, the $N(N-1)$ ``cherries'' each incur $\mathcal{O}(IR)$ memory, so the total storage scales as $\mathcal{O}\!\left(N(N-1)IR\right)$.
     
    \textbf{Computational complexity.} 
    (i) For each $k$ and $i$, updating $G_{k,i}$ $\left(i \neq k\right)$ in (\ref{eq:updateG}) involves the construction and solution of subproblems, which cost $\mathcal{O}(N^{2}(I^{2}R + I^{N-1}))$ and $\mathcal{O}(I^{N} + I^{2}R + IR^{3})$. 
    (ii) Updating $\mathcal{X}$ from (\ref{eq:updateX}) requires $\mathcal{O}(I^{N})$. 
    Therefore, the whole computational complexity at each iteration in  Algorithm~\ref{alg:iFCTN-TC} is $\mathcal{O}(N I^{N} + N^3 I^{N-1} + N^3 I^{2}R + N IR^{3})$. In practice, the problem always operates in the regime $I \gg R > N$.

\subsection{Convergence analysis}
    
    Under mild conditions, we aim to show that the sequence generated by Algorithm \ref{alg:iFCTN-TC} globally converges to a critical point. According to the finite length property \cite{attouch2013convergence}, establishing this global convergence requires the sequence to satisfy three key conditions: the Kurdyka-Łojasiewicz (KŁ) property, a sufficient decrease condition, and a relative error condition. We first formally establish these fundamental lemmas. The detailed proofs of these lemmas are provided in the appendix.

    \begin{lemma}\label{lemma:KL}
        The objective function $F(\mathcal{M})$ in problem (\ref{eq:iFCTN-TC}) is a Kurdyka-Łojasiewicz (KŁ) function.
    \end{lemma}

    Having confirmed the KŁ property of the objective function, we next characterize the dynamic behavior of the algorithmic sequence. 

    \begin{lemma}[Sufficient decrease]\label{lemma:Sufficient}
        Let $\left\{\mathcal{M}^{(s)}\right\}_{s \in \mathbb{N}}$ be a  sequence generated by Algorithm \ref{alg:iFCTN-TC}. Then, the sequence explicitly satisfies
        $$
        \begin{aligned}
            F(\mathcal{M}^{(s)}) - F(\mathcal{M}^{(s+1)}) \geq \frac{\rho}{2} \left\|\mathcal{M}^{(s+1)} - \mathcal{M}^{(s)} \right\|_{F}^{2},
        \end{aligned}
        $$
        where 
        $\left\|\mathcal{M}^{(s+1)}-\mathcal{M}^{(s)}\right\|_{F}^{2}= \left\|\mathcal{X}^{(s+1)} - \mathcal{X}^{(s)}\right\|_{F}^{2}$  $+\sum_{k=1}^{N}\sum_{i \neq k}^{N} \left\|\textbf{G}_{k,i}^{(s+1)}-\textbf{G}_{k,i}^{(s)}\right\|_{F}^{2}$.
    \end{lemma}

   While Lemma \ref{lemma:Sufficient} guarantees a sufficient reduction in the objective value, the KŁ framework also requires a connection between the sequence's step size and the stationarity of the objective. 
    
    \begin{lemma}[Relative Error Condition]
        Let $\left\{\mathcal{M}^{(s)}\right\}$ be a bounded sequence generated by Algorithm \ref{alg:iFCTN-TC}. Then, there exists $\mathcal P^{(s+1)}\in \partial F(\mathcal{M}^{(s+1)})$ such that
        $$\left\|\mathcal P^{(s+1)}\right\|_{F} \leq L_{F} \left\|\mathcal{M}^{(s+1)} - \mathcal{M}^{(s)}\right\|_{F},$$
        where $L_{F} = N(N-1)\rho + \sum_{k=1}^{N} \sum_{i \neq k}^{N} L_{k,i}$. 
        Here, $L_{k,i}$ denotes the gradient's Lipschitz constant of $F(G, \mathcal{X})$ with respect to $G_{k,i}$.
    \end{lemma}
    
    With the above three lemmas established, we are now ready to present the main global convergence theorem for our algorithm.  
    
    \begin{theorem}\label{Theorem}
        Suppose that the sequence $\left\{\left(\textbf{G}^{(s)}\right)\right\}$ obtained by the Algorithm \ref{alg:iFCTN-TC} is bounded, then it globally converges to a critical point of $F(\textbf{G},\mathcal{X})$.
    \end{theorem}

    \begin{proof}
         According to the finite length property \cite{attouch2013convergence}, the bounded sequence $\left\{\mathcal{M}^{(s)}\right\}_{s \in \mathbb{N}}$ could converge to the critical point when the following three lemmas are satisfied.
    \end{proof}

         \begin{figure}[ht]
        \centering
        \includegraphics[width=\textwidth]{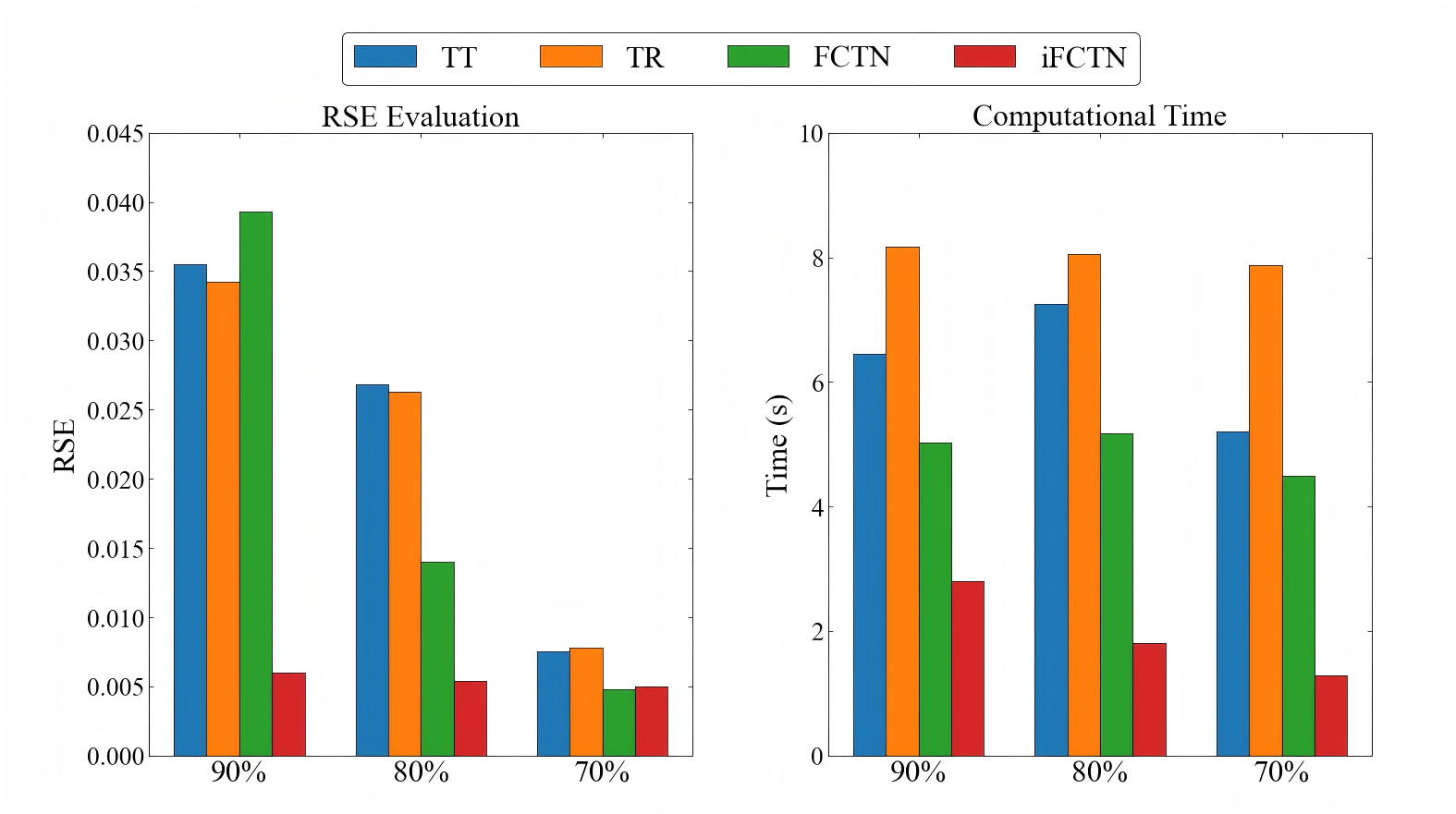}
        \caption{Results of RSE and runtime on the synthetic dataset with different missing rates.}
        \label{fig:synthetic}
    \end{figure}

      \begin{figure}[ht]
        \centering
        \includegraphics[scale = 0.17]{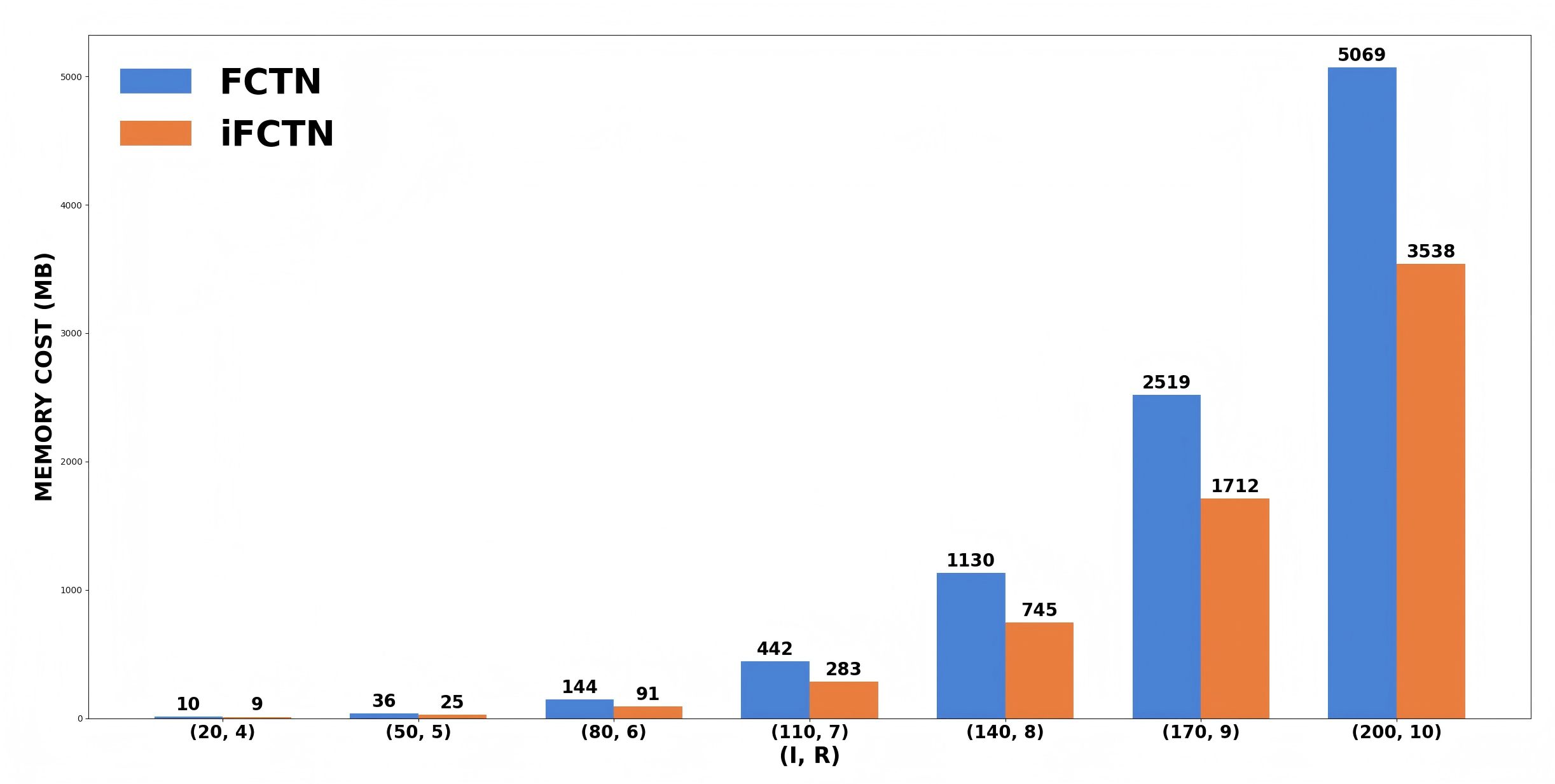}
        \caption{Memory cost of the FCTN and iFCTN under the fourth-order tensor with various sizes $I \times I \times I \times I$ and their ranks $(R, R, R, R, R, R)$.}
        \label{fig:memory}
    \end{figure}

            \begin{table}[ht]
    \centering
    \caption{Memory cost (MB) of methods under various sizes $I$ and the fixed rank R = 5 for the fourth-order tensor}
    \label{tab:synthetic1}
    \begin{tabular}{llllll}
    \toprule
    method & \multicolumn{5}{c}{$I$} \\
    \cmidrule(lr){2-6}
          & 20    & 40    & 60    & 80    & 100        \\
    \midrule
    FCTN  & 12.8  & 27.0  & 50.3  & 83.1  & 123.5 \\
    iFCTN  & 10.9  & 19.0  & 32.6  & 52.5   & 76.5\\
    \bottomrule
    \end{tabular}
    \end{table}

    \begin{table}[ht]
    \centering
    \caption{Memory cost (MB) of methods under various ranks $R$ and the fixed size $I = 40$ for the fourth-order tensor}
    \label{tab:synthetic2}
    \begin{tabular}{llllll}
    \toprule
    method & \multicolumn{5}{c}{$R$} \\
    \cmidrule(lr){2-6}
          & 5     & 10     & 15      & 20      & 25        \\
    \midrule
    FCTN  & 27.0  & 217.4  & 940.1  & 2844.7  & 6955.4 \\
    iFCTN  & 19.0  & 151.3  & 688.9  & 2108.6  & 5063.4\\
    \bottomrule
    \end{tabular}
    \end{table}

\section{Experiments}\label{sec:Experiments}

   Extensive numerical experiments were conducted to evaluate the effectiveness of the proposed iFCTN-TC method with different random-missing (RM) and fiber-missing (FM) levels. We compared iFCTN against eight state-of-the-art baselines: HaLRTC \cite{HaLRTC}, TNN \cite{TNN}, TMac \cite{TMac}, TFTC \cite{TFTC}, SNN~\cite{HaLRTC}, Tucker \cite{tucker1966some}, TT \cite{TT}, and FCTN \cite{FCTN}. The algorithm parameters were set as follows: the proximal parameter $\rho = 0.1$, maximum iteration $t_{\max}= 1000$, and stopping criteria $\epsilon = 10^{-5}$. To minimize the influence of the numerical optimizer, the TMac, TFTC, Tucker, TT, and FCTN models were also solved using the PAM scheme. For HaLRTC, the weight vector $\alpha$ was fixed to $(\frac{1}{N}, \dots, \frac{1}{N})$. The hyperparameters for those methods were tuned over prescribed ranges and selected manually for best performance. All datasets were pre-normalized to $[0,1]$.

   For evaluation, we report mean peak signal-to-noise ratio (MPSNR), mean structural similarity (MSSIM) \cite{PSNRSSIM}, relative error (RSE), i.e. $\text{RSE} = \left\|\mathcal{X}_{ture} - \mathcal{X}\right\|_{F}/\left\|\mathcal{X}\right\|_{F}$, the root mean square error (RMSE), i.e. $\text{RMSE} = \sqrt{\frac{1}{T}\left\| \mathcal{X}_{true} -\mathcal{X} \right\|_{F}^{2}}$, and total computational time (includes the time required to tune all parameters) are used to evaluate model performance. 
   
   All simulations are implemented in Pytorch (v2.0.0) and executed on a Windows 11 laptop with an Intel Core i7-12700H CPU @ 2.30 GHz, 32GB of RAM, and an NVIDIA GeForce RTX 3070 Laptop GPU.

\subsection{Synthetic Data Experiments}
    This section mainly aims to verify the superiority of the
    proposed iFCTN decomposition over the TT, TR and FCTN decompositions by contrasting the performance of their corresponding TC methods, i.e., iFCTN-TC, TT-TC, TR-TC and FCTN-TC. All methods are solved by PAM to get rid of the influence of the algorithm. We generate synthetic tensors by the tucker decomposition sampled from uniform distribution $U(0,1)$ The testing data includes a fourth-order tensor of size $20 \times 20 \times 20 \times 20$, whose Tucker rank is $(5,5,5,5)$. 
    
    As illustrated in Fig.~\ref{fig:synthetic}, while the RSE of TT, TR, and FCTN degrades significantly under extreme data missing rates (e.g., 80\% and 90\%), the proposed iFCTN remains remarkably stable and maintains a high level of precision. Notably, compared to its predecessor FCTN, iFCTN not only prevents the severe performance drop at the 90\% missing rate but also maintains comparable accuracy at lower missing rates. Furthermore, in terms of computational efficiency, iFCTN shows a dramatic advantage. It requires the least amount of time to execute (under 3 seconds), effectively cutting the time required by FCTN in half, whereas TR and TT take significantly longer.

    To further evaluate the storage efficiency, we compare the memory consumption of FCTN and iFCTN under different tensor sizes $I$ and ranks $R$ in Fig.~\ref{fig:memory}. As shown in the bar chart, iFCTN consistently requires less memory than the original FCTN across all tested configurations. More importantly, as the scale of $(I, R)$ increases, the memory-saving advantage of iFCTN becomes increasingly prominent. For instance, when $(I, R) = (200, 10)$, iFCTN reduces the memory cost from 5069 MB to 3538 MB, saving approximately 30\% of storage space.

    Table~\ref{tab:synthetic1} ($R=5$) shows that while memory consumption scales with dimension $I$, iFCTN maintains a significant advantage, reducing overhead by approximately 38\% at $I=100$ (76.5 MB vs. 123.5 MB). Furthermore, Table~\ref{tab:synthetic2} ($I=40$) reveals that memory cost is highly sensitive to the rank $R$. Specifically, in the high-rank case ($R=25$), iFCTN achieves a substantial saving of nearly 1.9 GB compared to FCTN (5063.4 MB vs. 6955.4 MB), demonstrating its robustness in memory-constrained environments.

          \begin{figure}[t]
        \centering
        \includegraphics[width=\textwidth]{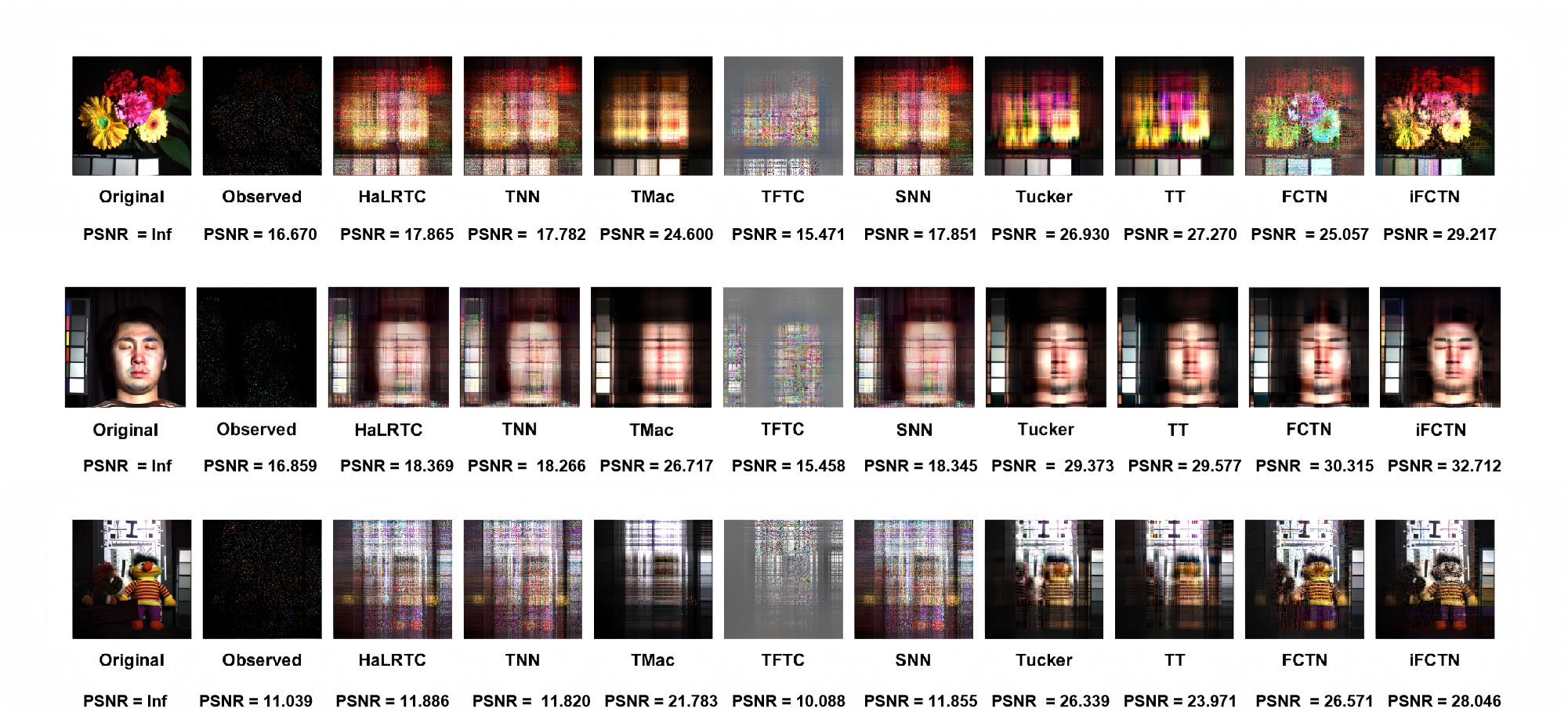}
        \caption{The pseudo-color renderings of reconstructions (composed of the 29th, 19th, and 9th bands) and corresponding PSNRs for MSI datasets. The first row is 
        \textit{flower} with RM = $90\%$, the second row is \textit{face} with RM = $80\%$, the third row is \textit{toy} with RM = $99\%$.}
        \label{fig:MSI_PSNR}
    \end{figure}

        \begin{table}[h]
\centering
\setlength{\tabcolsep}{3pt}
\renewcommand{\arraystretch}{1.15}
\footnotesize
\caption{Results of MPSNR, MSSIM, and runtime over 5 independent runs for  \textit{face and flower} under different fiber missing ratios (FR). The best results are highlighted in {\bf bold}; the second-best are \underline{underlined}.}
\label{tab:MSI_FM}
\resizebox{\textwidth}{!}{
\begin{tabular}{l l cc c cc c cc c cc c}
\toprule
\multirow{2}{*}{Dataset} & \multirow{2}{*}{method} &
\multicolumn{3}{c}{$80\%$} & \multicolumn{3}{c}{$70\%$} & \multicolumn{3}{c}{$60\%$} & \multicolumn{3}{c}{$50\%$} \\
\cmidrule(lr){3-5}\cmidrule(lr){6-8}\cmidrule(lr){9-11}\cmidrule(lr){12-14}
& & MPSNR & MSSIM & time (s) & MPSNR & MSSIM & time (s) & MPSNR & MSSIM & time (s) & MPSNR & MSSIM & time (s) \\
\midrule
\multirow{10}{*}{\textit{flower}}
& Observed & 17.439 & 0.484 & $-$ & 18.023 & 0.540 & $-$ & 18.690 & 0.589 & $-$ & 19.475 & 0.638 & $-$ \\
& HaLRTC   & 19.828 & 0.637 & 3.574 & 21.775 & 0.707 & 3.752 
 & 24.036 & 0.773 & 3.725 & 26.934 & 0.843 &3.769  \\
& TNN      & 19.681 & 0.623 & 338.992
 & 21.376 & 0.687 & 338.945 & 23.503 & 0.751 & 341.602  & 26.289 & 0.821 & 322.832  \\
& TMac    & 24.413 & 0.640 & 10.375 
 & 26.106 & 0.682 & 9.118  & 27.006 & 0.735 & 9.665  & 30.294 & 0.770 & 11.131 \\
& TFTC     & 25.811 & 0.625 & 6.433 & 26.598 & 0.671 & 3.482  & 27.451 & 0.705 & 2.403  & 28.335 & 0.748 & 1.423 \\
& SNN      & 20.043 & 0.628 & 39.284 
 & 21.928 & 0.697 & 	35.493  & 24.333 & 0.766 & 	38.266 & 27.513 & 0.840 &  	37.207 \\
& Tucker   & 26.116 & 0.648 & 3.612 & 26.858 & 0.691 & 	2.422 	 & 27.605 & 0.729 & 2.379 & 28.428 & 0.765 & 1.761 \\
& TT       & \uline{26.514} & \uline{0.661} & 46.972
 & \uline{27.359} & 0.707 & 25.432  & 28.165 & 0.746 & 18.461 	 & 29.005 & 0.781 & 11.250  \\
& FCTN     & 25.765 & 0.623 & 12.689 & 26.736 & \uline{0.708} & 18.487  & \uline{29.195} & \uline{0.749} & 10.936 & \uline{30.071} & \uline{0.776} & 13.245 \\
& iFCTN     & \textbf{27.248} & \textbf{0.703} &  7.518 
 & \textbf{28.987} & \textbf{0.758} & 3.740  & \textbf{30.243} & \textbf{0.803} & 2.669  & \textbf{31.196} & \textbf{0.836} & 	1.061  \\
\midrule
\multirow{10}{*}{\textit{face}}
& Observed & 17.628 & 0.414 & $-$ & 18.176 & 0.472 & $-$ & 18.902 & 0.526 & $-$ & 19.682 & 0.578 & $-$ \\
& HaLRTC   & 20.380 & 0.659 & 3.971 
 & 22.483 & 0.721 & 	3.858  & 25.451 & 0.789 & 3.924 & 29.222 & 0.856 & 3.954  \\
& TNN      & 20.158 & 0.653 & 346.060 & 22.105 & 0.712 &  	344.070 & 24.807 & 0.776 & 	333.544 & 28.143 & 0.838 & 313.699 \\
& TMac     & 19.284 & 0.526 & 18.603 & 22.094 & 0.656 & 15.272 & 25.459 & 0.740 & 11.887 & 29.734 & 0.826 & 17.876  \\
& TFTC     & \uline{29.882} & 0.778 & 30.023 & 31.812 & 0.819 & 19.671 & 31.828 & 0.845 & 8.550  & 32.640 & 0.873 & 3.764 \\
& SNN      & 17.960 & 0.507 & 42.769 & 22.754 & 0.726 & 43.725 & 25.878 & 0.796 & 41.526  & 29.899 & 0.865 & 40.294 \\
& Tucker   & 29.573 & 0.778 & 3.218 & 30.385 & 0.814 & 1.782 & 31.195 & 0.843 & 1.115 & 32.105 & 0.868 & 0.888 \\
& TT       & 29.872 & \uline{0.786} & 53.773 & 30.779 & 0.824 & 18.967 & 31.566 & 0.852 & 9.695  & 32.316 & 0.877 & 6.748  \\
& FCTN      & 28.465 & 0.782 & 10.148 & \uline{31.031} & \uline{0.814} & 10.368 & \uline{31.282} & \uline{0.837} & 6.683 & \uline{32.876} & \uline{0.873} & 7.424 \\
& iFCTN     & \textbf{31.243} & \textbf{0.832} & 7.948
 & \textbf{33.121} & \textbf{0.881} & 4.023 & \textbf{34.316} & \textbf{0.904} & 1.819  & \textbf{35.124} & \textbf{0.922} & 1.085 \\
\midrule
\multirow{10}{*}{\textit{toy}}
& Observed & 11.792 & 0.355 & $-$ & 12.384 & 0.417 & $-$ & 13.049 & 0.477 & $-$ & 13.836 & 0.533 & $-$ \\
& HaLRTC   & 13.294 & 0.519 & 3.981 & 14.649 & 0.586 & 	3.787  & 16.146 & 0.650 & 3.935  & 17.938 & 0.720 & 3.949\\
& TNN      & 13.211 & 0.514 & 340.771 & 14.349 & 0.580 & 	344.170  & 15.869 & 0.642 & 332.663 & 17.505 & 0.710 & 327.541 \\
& TMac     & 13.022 & 0.441 & 7.477 & 15.775 & 0.536 &  	10.515  & 19.326 & 0.646 & 8.693 & 23.441 & 0.730 & 11.164 \\
& TFTC     & \uline{22.861} & \uline{0.652} & 2.498 
 & 23.654 & 0.690 & 	10.961 & 24.420 & 0.726 & 4.208  & 25.286 & 0.759 & 3.782  \\
& SNN      & 13.403 & 0.520 & 8.552 & 14.626 & 0.588 & 	40.289  & 16.272 & 0.654 & 37.338 & 18.104 & 0.725 & 37.703 \\
& Tucker   & 20.451 & 0.626 & 5.086 & 24.287 & 0.714 & 2.247 & 25.212 & 0.753 & 4.307 	 & 26.151 & 0.787 & 1.979 \\
& TT       & 19.779 & 0.613 & 8.937 & \uline{24.470} & \uline{0.720} & 6.725 & 25.390 & 0.759 & 	1.785 & 26.350 & 0.793 & 4.659 \\
& FCTN     & 16.116 & 0.472 & 12.794 & 22.947 & 0.700 &  	13.299 & \uline{27.326} & \uline{0.790} & 10.341  & \uline{28.589} & \uline{0.825} & 10.374  \\
& iFCTN     & \textbf{23.204} & \textbf{0.667} & 7.164
 & \textbf{25.186} & \textbf{0.730} & 2.800 & \textbf{27.634} & \textbf{0.818} & 2.867  & \textbf{29.483} & \textbf{0.854} & 2.220 \\
\bottomrule
\end{tabular}}
\vspace{-0.6em}
\end{table}

    \begin{table}[ht]
\centering
\renewcommand{\arraystretch}{1.2}
\caption{Comparison of storage costs on the \textit{flower} dataset ($256 \times 256 \times 31$). The most compact representation is highlighted in \textbf{bold}.}
\label{tab:storage_cost}
\begin{tabular}{lccc}
\toprule
Method & Rank Configuration & Parameters & Storage Ratio (\%) \\
\midrule
Tucker & (8, 8, 8)        & 4,856  & 0.24\% \\
TT      & (9, 9)               & 23,319 & 1.15\% \\
FCTN  & (4, 4, 4)            & 8,688  & 0.45\% \\
iFCTN   & (4, 4, 4) & \textbf{4,344} & \textbf{0.21\%} \\
\bottomrule
\end{tabular}
\end{table}

\subsection{Multispectral images} 
    We first evaluate three MSI data (\textit{face, toy, flower}) from \textit{CAVE \footnote{http://www.cs.columbia.edu/CAVE/databases. The results for the \textit{toy} dataset are provided in the appendix due to space constraints.}}, each of size $256 \times 256 \times 31$ (spatial height $\times$ spatial width $\times$ bands). 
    For iFCTN-TC, following \cite{revealFCTN}, we set 
    $R_{1,2} \in \left\{3,6,9,15,30,36\right\}$ and $R_{1,3}= R_{2,3} \in \left\{3,6,9\right\}$. For FCTN-TC, we set
    $R_{1,2} \in \left\{2,3,6,9,15,30,36\right\} $ and $R_{1,3}= R_{2,3} \in \left\{3,6,9,12\right\}$. For TT-TC and Tucker-TC rank, we select ranks from  $\left\{2,3,6,9,12,15\right\}$. We report the best performance overall in rank selections.
   
    Fig.~\ref{fig:MSI_PSNR} presents the reconstruction results and the corresponding PSNRs for \textit{flower} and \textit{face}. We observe that iFCTN is robust and efficient. For instance, on both \textit{flower} and \textit{face}, its PSNR exceeds the second-best result by 2.5~dB. Due to space limitations, we place the results corresponding to \textit{Toy} in the supplementary materials.
    
    We continue to present the results over fiber missing in terms of bands in Table~\ref{tab:MSI_FM}. The PSNR and SSIM outcomes are similar to those shown in  Fig.~\ref{fig:MSI_PSNR}. For instance, iFCTN-TC exceeds the second-best method by 3~dB in PSNR at FM = $60\%$ for \textit{face}, and by 1.5~dB in PSNR at FM = $70\%$ for \textit{flower}.      

    We evaluate the storage cost of several decomposition methods---Tucker, TT, FCTN, and iFCTN---on the \textit{flower}. The numerical results are summarized in Table~\ref{tab:storage_cost}. The iFCTN decomposition involves 4,344 parameters, which corresponds to 0.21\% of the total tensor elements. Under the specified rank configurations, the parameter counts for the TT and FCTN decompositions are 23,319 (1.15\%) and 8,688 (0.45\%), respectively, while the Tucker decomposition requires 4,856 parameters (0.24\%). These data points reflect the storage requirements of each method for the given dataset.

    \begin{figure}[ht]
        \centering
    \includegraphics[width=\textwidth]{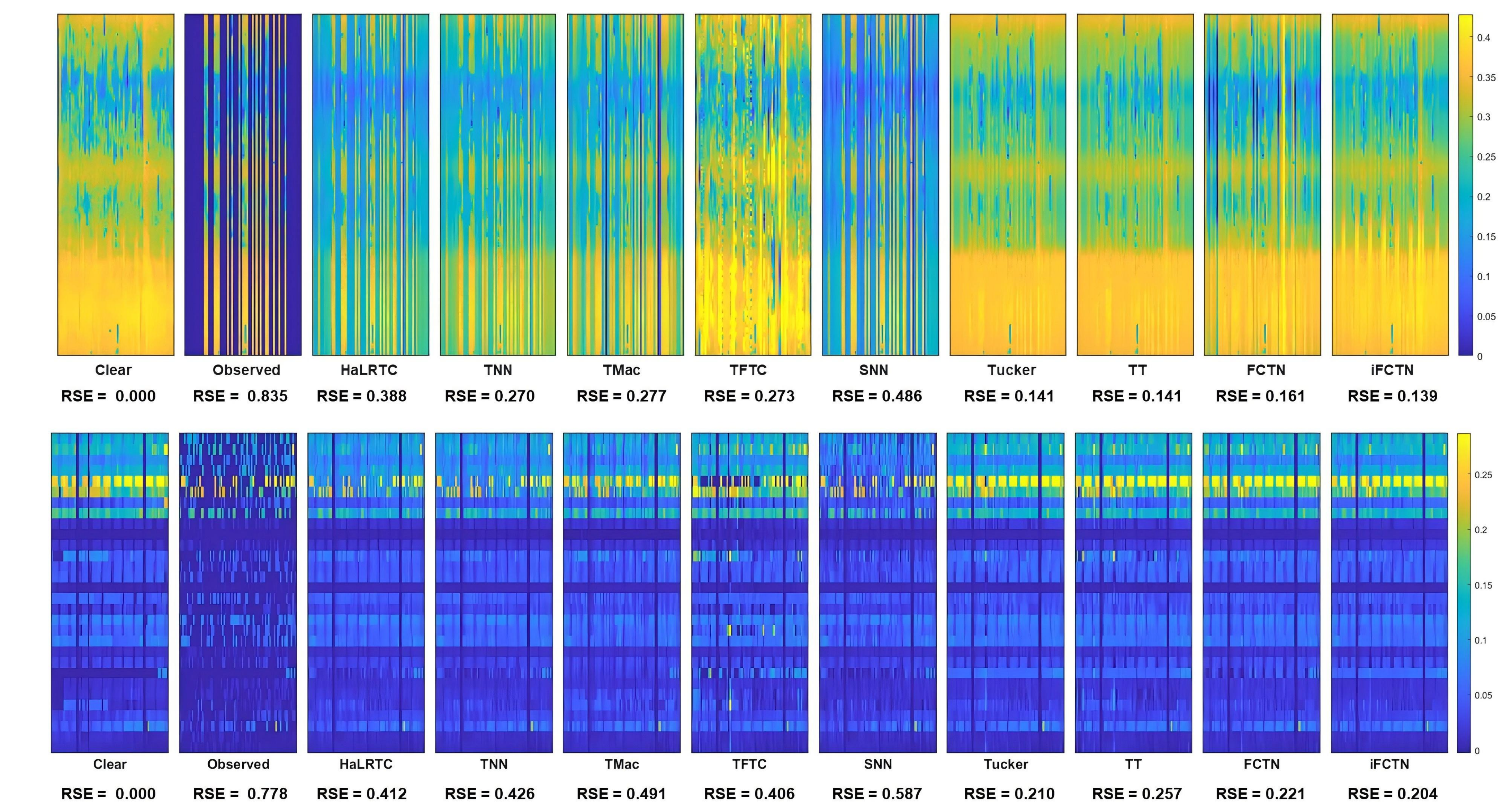}
        \caption{Reconstructions on Day 1. The first row is \textit{Guangzhou} with FM = $70\%$ and the second row is \textit{Birmingham} with FM = $60\%$.}
        \label{fig:tarffic_FM}
    \end{figure}

    \begin{table}[h]
\centering
\setlength{\tabcolsep}{4pt}
\renewcommand{\arraystretch}{1.15}
\footnotesize
\caption{Performance on RSE, RMSE, and runtime over 5 independent runs for \textit{Guangzhou} and \textit{Birmingham} under different random missing ratios. The best results are highlighted in {\bf bold}; the second-best are \underline{underlined}.}
\label{tab:traffic_RM}
\resizebox{\textwidth}{!}{
\begin{tabular}{l l cc c cc c cc c cc c}
\toprule
\multirow{2}{*}{Dataset} & \multirow{2}{*}{method} & \multicolumn{3}{c}{$99\%$} & \multicolumn{3}{c}{$95\%$} & \multicolumn{3}{c}{$90\%$} & \multicolumn{3}{c}{$80\%$} \\
\cmidrule(lr){3-5}\cmidrule(lr){6-8}\cmidrule(lr){9-11}\cmidrule(lr){12-14}
& & RSE & RMSE & time (s) & RSE & RMSE & time (s)& RSE & RMSE & time (s)& RSE & RMSE & time (s) \\
\midrule
\multirow{10}{*}{\textit{Guangzhou}} & Observed & 0.995 & 0.337 & $-$ & 0.975 & 0.337 & $-$ & 0.949 & 0.337 & $-$ & 0.894 & 0.337 & $-$ \\
 & HaLRTC & 0.975 & 0.336 & 2.852 & 0.855 & 0.333 & 2.749 & 0.729 & 0.330 & 2.959 & 0.582 & 0.325 & 3.199 \\
 & TNN & 0.351 & 0.119 & 310.606 & 0.162 & 0.056 & 	300.348 & 0.144 & 0.051 & 294.694  & 0.129 & 0.049 & 284.548 \\
 & TMac & 0.691 & 0.234 & 62.824 & 0.174 & 0.063 & 44.175   & 0.145 & 0.051 & 39.911  & 0.123 & 0.053 & 37.639 \\
 & TFTC & 0.970 & 0.329 & 129.794 & 1.060 & 0.367 & 125.926   & 1.046 & 0.372 & 132.823 & 0.711 & 0.268 & 140.939 \\
 & SNN & 0.993 & 0.336 & 28.526 & 0.964 & 0.334 & 31.338  & 0.932 & 0.331 & 28.632 & 0.869 & 0.327 & 29.356 \\
 & Tucker & 0.229 & 0.078 & 101.479 & \uline{0.141} & \uline{0.049} & 93.766 & \uline{0.133} & \uline{0.047} & 51.052 & 0.117 & 0.044 & 24.524 \\
 & TT & 0.346 & 0.117 & 190.311 & 0.142 & 0.055 & 182.771  & 0.137 & 0.048 & 168.675  & 0.129 & 0.048 & 69.809 \\
 & FCTN & \uline{0.194} & \uline{0.066} & 43.768 & 0.152 & 0.052 & 48.389  & 0.150 & 0.053 & 39.224  & \uline{0.106} & \uline{0.040} & 42.351 \\
 & iFCTN & \textbf{0.178} & \textbf{0.060} &  10.294
 & \textbf{0.116} & \textbf{0.042} & 8.787 & \textbf{0.109} & \textbf{0.040} & 3.634  & \textbf{0.101} & \textbf{0.039} &  2.210 \\
\midrule
\multirow{10}{*}{\textit{Birmingham}} & Observed & 0.996 & 0.197 & $-$ & 0.970 & 0.195 & $-$ & 0.946 & 0.195 & $-$ & 0.893 & 0.193 & $-$ \\
 & HaLRTC & 0.996 & 0.197 & 0.284 & 0.730 & 0.147 & 0.902 & 0.394 & 0.081 & 0.936 & 0.774 & 0.033 & 0.906 \\
 & TNN & 0.996 & 0.197 & 21.911 & 0.462 & 0.093 & 21.291 & 0.250 & 0.052 & 23.056 & 0.154 & 0.033 & 21.144 \\
 & TMac & 1.021 & 0.202 & 2.461 & 0.455 & 0.091 & 3.264  & 0.272 & 0.056 & 4.305  & \uline{0.149} & \uline{0.032} & 	3.710 \\
 & TFTC & 0.992 & 0.196 & 13.733 & 0.965 & 0.194 & 16.448  & 1.111 & 0.229 & 17.939  & 1.164 & 0.282 & 17.331 \\
 & SNN & 0.993 & 0.196 & 1.278 & 0.480 & 0.097 & 7.852 & 0.235 & 0.048 & 8.307  & 0.558 & 0.121 & 6.064 \\
 & Tucker & \textbf{0.722} & \textbf{0.143} & 3.878
 & 0.414 & 0.083 & 3.913 & 0.230 & 0.047 & 1.865  & 0.170 & 0.037 & 3.902 \\
 & TT & 0.805 & 0.159 & 11.532 & \uline{0.281} & \uline{0.056} & 11.543  & \uline{0.189} & \uline{0.039} & 6.622 & 0.175 & 0.038 & 1.820 \\
 & FCTN & 0.800 & 0.158 & 4.032 & 0.430 & 0.086 & 3.906 & 0.239 & 0.049 & 4.261  & 0.156 & 0.034 & 3.194 \\
 & iFCTN & \uline{0.739} & \uline{0.146} & 3.521
 & \textbf{0.192} & \textbf{0.039} & 3.613  & \textbf{0.177} & \textbf{0.036} & 1.674 & \textbf{0.145} & \textbf{0.031} & 0.516\\
\bottomrule
\end{tabular}}
\end{table}
    
    \subsection{Traffic data} 
    We further evaluate iFCTN-TC on two traffic datasets: 
    \begin{itemize}
    \item \textit{Guangzhou\footnote{https://doi.org/10.5281/zenodo.1205229}}     contains travel speed observations from 214 road segments in two months (61 days from August 1, 2016 to September 30, 2016) at 10 min interval (144 time intervals in a day), whose size is $214 \times 61 \times 144$ (road segment $\times$ day $\times$ time interval); 
    \item \textit{Birmingham\footnote{https://archive.ics.uci.edu/dataset/482/parking+birmingham}} is collected from 18 Car park ID in two months (October 4,2016 to December 19, 2016) at 30 min interval, whose size is $30 \times 77 \times 18$ (time interval $\times$ day $\times$ zone).
    \end{itemize}

    Table~\ref{tab:traffic_RM} and Fig.~\ref{fig:tarffic_FM} present the numerical and visual reconstruction results for \textit{Guangzhou} and \textit{Birmingham} under different scenarios. We observe that iFCTN is robust and efficient. For instance, on \textit{Guangzhou} at RM = $99\%$, its RSE improves upon the second-best result by 0.016, while requiring only one-fifth of the computation time compared to the standard FCTN. Furthermore, under severe structural missing conditions (\textit{Guangzhou} at FM = $70\%$ and \textit{Birmingham} at FM = $60\%$), iFCTN successfully restores complex traffic patterns, outperforming the second-best RSE results by 0.002 and 0.006, respectively.

\section{Conclusion}
    We have presented intra-block FCTN (iFCTN), an (un)folding-free reformulation of the fully connected tensor network decomposition. By decomposing the tensor factors into matrix factors with the KR product, iFCTN eliminates the expensive auxiliary subnetworks reconstruction that burdens standard FCTN. When applied to the TC problem, the proposed PAM algorithm exhibits provable convergence,  along with favorable reconstruction accuracy and robust performance. 
    Extensive experiments on public data show that iFCTN delivers state-of-the-art recovery results. 
    
     {\bf Limitations.} Similar to CP, TR, and FCTN decompositions \cite{kolda2009tensor, TR, FCTN}, the proposed iFCTN decomposition faces a difficult rank-selection problem. In the future, we will develop principled strategies for rank determination and extend iFCTN to online and large-scale   scenarios.

\section*{Declaration of competing interest}
The authors declare that they have no known competing financial interests or personal relationships that could have appeared to
influence the work reported in this paper.

\section*{Acknowledgments}
This research is supported by NSFC (Nos. 12471282  and 12131004).

\section*{Data availability}
Data are publicly available.

\clearpage
\appendix
\section*{Appendix}
\setcounter{equation}{0}
\setcounter{proposition}{0}
\setcounter{theorem}{0}
\setcounter{figure}{0}
\setcounter{table}{0}
\setcounter{lemma}{0}



\section{Proof of Proposition 1}
\label{sec: Proposition}

\begin{proposition}[\textbf{Mode-$k$   Cherry Sub-network}]\label{propos:cherrysubnet}
        Suppose that $\mathcal{X} = \operatorname{iFCTN} \left(\left\{\textbf{G}_{t}\right\}_{t=1}^{N}\right)$. 
        Then, the cherry sub-network $\textbf{Z}_{k} \in \mathbb R^{\prod_{i \in [N] \setminus k} I_{i} \times \prod_{i \in [N] \setminus k} R_{k,i}}$  comprises two components: (i) the cherry factors $\{\mathbf{G}_{t,i}\}_{i\neq t}$ for all  $t\in[N]\setminus k$; (ii) the   cherry factors $\textbf{G}_{k}$. 
        Specifically, 
        \begin{align}\label{equ:Zk}
             \textbf{Z}_{k} = & \mathrm{diag}\left(\mathrm{vec}\left(\mathcal{S}_{k}\right)
             \right) \times  \mathop{\otimes} \limits_{i \in \left\{N,\dots,1\right\} \setminus k} \textbf{G}_{i,k}^{\top},
        \end{align}
            where $\mathcal{S}_{k}$ is an $(N-1)$ th-order tensor shown as follows,
        \begin{align*}
             \mathcal{S}_{k} = & \mathop{\circ} \limits_{i<j, i, j \in [N] \setminus k} \left\{\operatorname{reshape}\left(\mathbf{G}_{i,j}^{\top}\mathbf{G}_{j,i},[1,\dots,I_{i},\dots,I_{j},1,\dots,1]\right)\right\}.  
     \end{align*}
     Moreover, the mode-$k$ unfolding of $\mathcal{X}$ can be expressed by:  
        \begin{equation}
            \left(\mathcal{X}_{\left(k\right)}\right)^{\top} = \textbf{Z}_{k} \left(\mathcal{G}_{(k)}\right)^{\top},
        \end{equation}
        where $\mathcal{G}_{(k)}=\mathcal{CT}_k\left(\{{\bf G}_i\}_{i=1,i\neq k}^N\right).$
     \end{proposition}

     \begin{proof}
         According to the definition of cherry product, $\mathcal{Z}_{k} = \operatorname{iFCTN}(\textbf{G}_{1},\dots,\textbf{G}_{k-1},\textbf{G}_{k+1},\dots,\textbf{G}_{N})$ is of size $(I_{1}, R_{1,k}, \dots, I_{k-1}, R_{k-1,k}, R_{k+1,k}, I_{k+1}, \dots, R_{N,k}, I_{N})$ and is given by
         \begin{equation}
             \begin{aligned}
                 & \mathcal{Z}_{k}(i_1, r_{1,k}, \dots, i_{k-1}, r_{k-1,k}, r_{k+1,k}, i_{k+1},\dots, r_{N,k}, i_N) \\
                 & = \prod_{m=1}^{k-1}\textbf{G}_{m,k}(r_{m,k}, i_{m}) \prod_{n = k+1}^{N}\textbf{G}_{n,k}(r_{n,k}, i_{n}) \left(\sum_{j \in [N] \setminus k} \sum_{i<j} \textbf{G}_{i,j}^{\top} \textbf{G}_{j,i} \right)_{I_i,I_j}\\
                 & = \prod_{m=1}^{k-1}\textbf{G}_{m,k}(r_{m,k}, i_{m}) \prod_{n = k+1}^{N}\textbf{G}_{n,k}(r_{n,k}, i_{n}) \mathcal{S}_{k} (i_{1:k-1}, i_{k+1:N}), 
             \end{aligned}
         \end{equation}
         where $\mathcal{S}_{k}(i_{1:k-1}, i_{k+1:N}) \in \mathbb{R}^{\prod_{i \neq k} {I_{i}} \times \prod_{i \neq k} {I_{i}}}$ is defined as
          \begin{align*}
             \mathcal{S}_{k} = & \mathop{\circ} \limits_{i<j, i, j \in [N] \setminus k} \left\{\operatorname{reshape}\left(\mathbf{G}_{i,j}^{\top}\mathbf{G}_{j,i},[1,\dots,I_{i},\dots,I_{j},1,\dots,1]\right)\right\}.  
        \end{align*}
        
         Perform a generalized matricization of the tensor $\mathcal{Z}_{k}$ to obtain $\textbf{Z}_{k}^{\top}$, in which all modes $\left\{R_{i}\right\}_{i \neq k}$  are arranged row-wise and all modes $\left\{I_{i}\right\}_{i \neq k}$ are column-wise. 
        \begin{equation}
            \begin{aligned}
                & \textbf{Z}_{k}^{\top} \left((r_{1:k-1,k}, r_{k+1:N,k}), (I_{1:k-1}, I_{k+1:N})\right) \\
                & =  \mathcal{Z}_{k} (i_1, r_{1,k}, \dots, i_{k-1}, r_{k-1,k}, r_{k+1,k}, i_{k+1},\dots, r_{N,k}, i_N).
            \end{aligned}
        \end{equation}
        Since the Kronecker product $\otimes$ naturally captures the structure of this combined indexing, we have
        \begin{equation}
            \textbf{K}(:,col(i_{1:k-1}, i_{k+1:N}) = \mathop{\otimes} \limits_{i \in \left\{N,\dots,1\right\} \setminus k} \textbf{G}_{i,k}.
        \end{equation}
        This implies that in the matrix $\textbf{Z}_{k}^{\top}$, the column indexed by $(I_1, \dots, I_{k-1}, I_{k+1}, \dots, I_{N})$ can be expressed as:
        \begin{equation}
            \begin{aligned}
                & \textbf{Z}_{k}^{\top}(:, (i_{1:k-1}, i_{k+1:N}))  = \mathcal{S}_{k} (i_{1:k-1}, i_{k+1:N})) \times \textbf{K}(:,col(i_{1:k-1}, i_{k+1:N})
            \end{aligned}
        \end{equation}
        Therefore, the entire matrix $\textbf{Z}_{k}^{\top}$ can be expressed as:
        \begin{equation}
            \textbf{Z}_{k}^{\top} = \mathop{\otimes} \limits_{i \in \left\{N,\dots,1\right\} \setminus k} \textbf{G}_{i,k} \times  \mathrm{diag}\left(\mathrm{vec}\left(\mathcal{S}_{k}\right)
             \right),
        \end{equation}
        Namely, 
        \begin{align}\label{equ:Zk}
             \textbf{Z}_{k} = & \mathrm{diag}\left(\mathrm{vec}\left(\mathcal{S}_{k}\right)
             \right) \times  \mathop{\otimes} \limits_{i \in \left\{N,\dots,1\right\} \setminus k} \textbf{G}_{i,k}^{\top} \nonumber,
        \end{align}
        
        Given the expression for $\textbf{Z}_{k}$, it is straightforward to obtain the expansion for the mode-k unfolding of $\mathcal{X}$:
         \begin{align}
            \left(\mathcal{X}_{\left(k\right)}\right)^{\top} = \textbf{Z}_{k} \left(\mathcal{G}_{(k)}\right)^{\top}. \nonumber
         \end{align}
         where $\mathcal{G}_{(k)}=\mathcal{CT}_k\left(\{{\bf G}_i\}_{i=1,i\neq k}^N\right).$
         
         The proof is completed.
     \end{proof}

\subsection{Proof of Proposition 2}
\begin{proposition}[\textbf{Uniqueness}]
     Let $\mathcal{X}$ be an $N$-th order tensor admitting an iFCTN decomposition with rank $\mathbf{R}$. The cherry factors $\{M_{k,j}\}_{j \neq k}$ in iFCTN are essentially unique (up to permutation and scaling), provided that Kruskal's condition is satisfied. Specifically, for any mode $k$, if the sum of k-ranks of the cherry factors satisfies:
    $$\sum_{j \neq k} k_{\text{rank}}(M_{k,j}) \ge  I_k + N-1,$$ 
    then the factors $\{M_{k,j}\}$ are unique up to a common permutation matrix $\mathbf{\Pi}$ and scaling diagonal matrices $\{\mathbf{\Lambda}_j\}$ such that $\prod \mathbf{\Lambda}_j = \mathbf{I}$.
\end{proposition}

\begin{proof}
   Recall the definition of the iFCTN core tensor $\mathcal{G}_k$ for a specific mode $k$. Its mode-$k$ unfolding is defined as the KR product of the cherry factors:
    $$
    (\mathcal{G}_{k})_{(k)}^\top = \mathbf{M}_{k,N} \odot \dots \odot \mathbf{M}_{k,k+1} \odot \mathbf{M}_{k,k-1} \odot \dots \odot \mathbf{M}_{k,1},
    $$
    where $\mathbf{M}_{k,j} \in \mathbb{R}^{R_{k,j} \times I_k}$. Let $\mathbf{m}_{k,j}^{(t)}$ denote the $t$-th column of the matrix $M_{k,j}$. The tensor $\mathcal{G}_k$ can be explicitly written as a sum of rank-1 tensors:
    $$
    \mathcal{G}_{k} = \sum_{t=1}^{I_{k}} \mathbf{m}_{k,1}^{(t)} \circ \dots \circ \mathbf{e}_{t} \circ \dots \circ \mathbf{m}_{k,N}^{(t)},
    $$
    where $\mathbf{e}_t \in \mathbb{R}^{I_k}$ is the standard basis vector (the $t$-th column of the identity matrix $I_{I_k}$).
    This formulation confirms that $\mathcal{G}_k$ admits a CP decomposition with rank $I_k$, where the factor matrix corresponding to mode $k$ is fixed as the identity matrix $\mathbf{I}_{I_k}$. 
    
    According to the celebrated Kruskal's Theorem \cite{kolda2009tensor}] for the uniqueness of CP decomposition, a sufficient condition for the factors to be essentially unique is based on the k-rank of the factor matrices. For the core tensor $\mathcal{G}_k$, the factors are $\{\mathbf{M}_{k,j}\}_{j \neq k}$ and the identity matrix. Since its full rank, the condition for uniqueness becomes:
    $$k_{\text{rank}}(\mathbf{I}) + \sum_{j \neq k} k_{\text{rank}}(M_{k,j}) \ge 2 \cdot \text{rank}(\mathcal{G}_k) + (N) - 1.$$
    Substitute $\text{rank}(\mathcal{G}_k) = I_k$:
    $$I_k + \sum_{j \neq k} k_{\text{rank}}(M_{k,j}) \ge 2 I_k + N - 1.$$
    Simplifying this inequality yields:$$\sum_{j \neq k} k_{\text{rank}}(M_{k,j}) \ge I_k + N - 1.$$
\end{proof}

\subsection{Proof of Proposition 3}
\begin{proposition}[Transposition Invariance]
    Supposing that an $N$th-order tensor $\mathcal{X} \in \mathbb{R}^{I_1 \times \dots \times I_N}$ has the following iFCTN decomposition parameterized by the set of cherry factors $\mathbf{M} = \{\mathbf{M}_{k,j} \mid 1 \le k,j \le N, k \neq j\}$:
    $$\mathcal{X} = \text{iFCTN}(\{\mathbf{M}_{k,j}\}_{k,j}).$$
    Then, its vector $\mathbf{n}$-based generalized tensor transposition $\vec{\mathcal{X}}^{\mathbf{n}}$ can be expressed as:
    $$\vec{\mathcal{X}}^{\mathbf{n}} = \text{iFCTN}(\{\mathbf{M}^{\mathbf{n}}_{k,j}\}_{k,j}),$$
    where $\mathbf{n}=(n_{1},n_{2},\cdot\cdot\cdot,n_{N})$ is a reordering of the vector $(1,2,...,N)$, and the new factors are given by the mapping $\mathbf{M}^{\mathbf{n}}_{k,j} = \mathbf{M}_{n_k, n_j}$.
\end{proposition}

\begin{proof}
    Let the index vector of the original tensor be $(i_1, i_2, \dots, i_N)$. The element-wise definition of the iFCTN decomposition for $\mathcal{X}$ is given by contracting the cherry factors:
    $$\mathcal{X}(i_1, i_2, \dots, i_N) = \sum_{r_{1,2}=1}^{R_{1,2}} \dots \sum_{r_{N-1,N}=1}^{R_{N-1,N}} \prod_{k=1}^{N} \left( \bigodot_{j \neq k} \mathbf{m}_{k,j}^{(r_{k,j})} \right)_{i_k},$$
    where $\mathbf{m}_{k,j}^{(r_{k,j})}$ represents the column of factor $M_{k,j}$ corresponding to the rank index $r_{k,j}$.
    
    The entry of the transposed tensor $\vec{\mathcal{X}}^{\mathbf{n}}$ at indices $(j_1, j_2, \dots, j_N)$ corresponds to the entry of the original tensor at the permuted indices:
    $$\vec{\mathcal{X}}^{\mathbf{n}}(j_1, j_2, \dots, j_N) = \mathcal{X}(j_{p_1}, j_{p_2}, \dots, j_{p_N}),$$
    where the indices are mapped such that if the $k$-th dimension of $\vec{\mathcal{X}}^{\mathbf{n}}$ corresponds to the $n_k$-th dimension of $\mathcal{X}$. 
    Thus, we set the inputs to $\mathcal{X}$ as $i_{n_k} = j_k$.Expanding the summation for the permuted indices:
    $$
    \begin{aligned}
        \vec{\mathcal{X}}^{\mathbf{n}}(j_1, \dots, j_N) &= \mathcal{X}(j_{n^{-1}_1}, \dots, j_{n^{-1}_N}) \\
        &= \sum_{r_{n_1,n_2}} \dots \sum_{r_{n_{N-1},n_N}} \prod_{k=1}^{N} \left( \bigodot_{t \neq n_k} \mathbf{m}_{n_k,t}^{(r_{n_k,t})} \right)_{j_k}.
    \end{aligned}
    $$
    
    In the original tensor, mode $n_k$ interacts with mode $n_l$ via the factor pair $(M_{n_k, n_l}, M_{n_l, n_k})$. In the transposed tensor $\vec{\mathcal{X}}^{\mathbf{n}}$, these modes are now at positions $k$ and $l$. We define the new factor set $\mathbf{M}^{\mathbf{n}}$ such that the factor connecting mode $k$ and mode $l$ in the new tensor :$$M^{\mathbf{n}}_{k,l} := M_{n_k, n_l}.$$
    
    Substituting this into the expansion, the interaction between the $k$-th and $l$-th dimension of $\vec{\mathcal{X}}^{\mathbf{n}}$ is governed by the rank indices $r_{n_k, n_l}$. By renaming the rank summation indices $r'_{k,l} = r_{n_k, n_l}$, the expression becomes:
    $$\vec{\mathcal{X}}^{\mathbf{n}}(j_1, \dots, j_N) = \sum_{r'_{1,2}} \dots \sum_{r'_{N-1,N}} \prod_{k=1}^{N} \left( \bigodot_{l \neq k} \mathbf{m}^{\mathbf{n}, (r'_{k,l})}_{k,l} \right)_{j_k}.$$
    This matches the definition of an iFCTN decomposition on $\vec{\mathcal{X}}^{\mathbf{n}}$ using the permuted factors. Thus, the permuted set $\{M_{n_k, n_l}\}$ remains a valid fully-connected set of factors for the new order.
\end{proof}

\subsection{Proof of Theorem 1}
\label{sec:Theorem}

Before proving  Theorem \ref{Theorem}, we first rewrite the proposed iFCTN-TC model as follows,
    \begin{equation}\label{eq:iFCTN-TC}
        \min _{\mathcal{M}} F(\mathcal{M}) = f(\mathcal{M}) + h(\mathcal{X}),
    \end{equation}
    where $\mathcal{M} =\left(\textbf{G}, \mathcal{X}\right)$,  $f(\mathcal{M})  =  \frac{1}{2}\left\|\mathcal{X}-\operatorname{iFCTN}\left(\textbf{G}_{1},\dots, \textbf{G}_{N}\right)\right\|_{F}^{2}$, and $h(\mathcal{X}) = \iota (\mathcal{X})$.

    \begin{theorem}\label{Theorem}
        Suppose that the sequence $\left\{\left(\textbf{G}^{(s)},\mathcal{X}^{(s)}\right)\right\}$ obtained by the Algorithm \ref{alg:iFCTN-TC} is bounded, then it globally converges to a critical point 
        $(\textbf{G}^{*}, \mathcal{X}^{*}) $ of $F(\textbf{G},\mathcal{X})$.
    \end{theorem}

    
    \begin{lemma}\label{lemma:KL}
        The objective function $F(\mathcal{M})$ in problem (\ref{eq:iFCTN-TC}) is a Kurdyka-Łojasiewicz (KŁ) function.
    \end{lemma}

    \begin{proof}
        Following \cite{attouch2013convergence,KL}, since $f(\mathcal{M})$ is a polynomial function of $N(N-1)+1$ coupling variations, it is an obviously real-analytic function. Regarding  $h(\mathcal{X})$, since the constraint set $\left\{\mathcal{X}: \mathcal{P}_{\Omega}(\mathcal{X})=\mathcal{P}_{\Omega}(\mathcal{O})\right\}$ is semi-algebraic, $\iota(\mathcal{X})$ is a semi-algebraic function. Therefore, $F(\mathcal{M})$, as a finite sum of real-valued analytic and semi-algebraic functions, is a  KŁ function.
    \end{proof}
        
    \begin{lemma}\label{lemma:Sufficient}[Sufficient decrease]
        Let $\left\{\mathcal{M}^{(s)}\right\}_{s \in \mathbb{N}}$ be a  sequence generated by Algorithm \ref{alg:iFCTN-TC}. Then, the sequence explicitly satisfies
        $$
        \begin{aligned}
            F(\mathcal{M}^{(s)}) - F(\mathcal{M}^{(s+1)}) \geq \frac{\rho}{2} \left\|\mathcal{M}^{(s+1)} - \mathcal{M}^{(s)} \right\|_{F}^{2},
        \end{aligned}
        $$
        where 
        $\left\|\mathcal{M}^{(s+1)}-\mathcal{M}^{(s)}\right\|_{F}^{2}= \left\|\mathcal{X}^{(s+1)} - \mathcal{X}^{(s)}\right\|_{F}^{2}$  $+\sum_{k=1}^{N}\sum_{i \neq k}^{N} \left\|\textbf{G}_{k,i}^{(s+1)}-\textbf{G}_{k,i}^{(s)}\right\|_{F}^{2}$.
    \end{lemma}

    \begin{proof}
       By the update scheme of $\left\{\textbf{G}_{k,i}^{(s)}\right\}$, $k \in [N]$ and $i \in [N] \setminus k$, we have
        \begin{equation}\label{G_ineq}
            \begin{aligned}
            & f\left(\mathbf{G}_{<k,i}^{(s+1)},\mathbf{G}_{k,i}^{(s+1)},\mathbf{G}_{>k,i}^{(s)},\mathcal{X}^{(s)}\right) + \frac{\rho}{2}\left\|\textbf{G}_{k,i}^{(s+1)}-\textbf{G}_{k,i}^{(s)}\right\|_{F}^{2}  
            \leq  f\left(\mathbf{G}_{<k,i}^{(s+1)},\mathbf{G}_{k,i}^{(s)},\mathbf{G}_{>k,i}^{(s)},\mathcal{X}^{(s)}\right). 
            \end{aligned}
        \end{equation}
        Likewise, it follows from the update scheme of $\mathcal X$ that 
        \begin{equation}\label{X_ineq}
            \begin{aligned}
                & f\left(\textbf{G}^{(s+1)},\mathcal{X}^{(s+1)}\right) + \frac{\rho}{2} \left\|\mathcal{X}^{(s+1)}-\mathcal{X}^{(s)}\right\|_{F}^{2}
                 \leq f\left(\textbf{G}^{(s+1)},\mathcal{X}^{(s)}\right).  
            \end{aligned}
        \end{equation}
        Finally, by taking summation over \eqref{G_ineq} and \eqref{X_ineq} and eliminating the duplicates, we can deduce 
        \begin{equation*}
            \begin{aligned}
                 & F(\mathcal{M}^{(s)}) - F(\mathcal{M}^{(s+1)}) 
                 \geq  \frac{\rho}{2} \left\|\mathcal{X}^{(s+1)} - \mathcal{X}^{(s)}\right\|_{F}  +  \sum_{k=1}^{N}\sum_{i \neq k}^{N} \frac{\rho}{2}\left\|\textbf{G}_{k,i}^{(s+1)}-\textbf{G}_{k,i}^{(s)}\right\|_{F}^{2}.
            \end{aligned}
        \end{equation*}
        The proof is completed.
    \end{proof}

    \begin{lemma}[Relative Error Condition]
        Let $\left\{\mathcal{M}^{(s)}\right\}$ be be a bounded sequence generated by Algorithm \ref{alg:iFCTN-TC}. Then, there exists $\mathcal P^{(s+1)}\in \partial F(\mathcal{M}^{(s+1)})$ such that
        $$\left\|\mathcal P^{(s+1)}\right\|_{F} \leq L_{F} \left\|\mathcal{M}^{(s+1)} - \mathcal{M}^{(s)}\right\|_{F},$$
        where $L_{F} = N(N-1)\rho + \sum_{k=1}^{N} \sum_{i \neq k}^{N} L_{k,i}$. 
        Here, $L_{k,i}$ denotes the gradient's Lipschitz constant of $F(G, \mathcal{X})$ with respect to $G_{k,i}$.
    \end{lemma}

    \begin{proof}
        According to the proof of Lemma \ref{lemma:Sufficient}, $\textbf{G}^{(s+1)}$ and $\mathcal{X}^{(s+1)}$ are the minimum solutions, thus we have $h(\mathcal{X}^{(s+1)}) \equiv 0$ for avoiding $F(\mathcal{M}) \rightarrow 0$. According to the fact that minimum solutions must satisfy the first-order optimal conditions, i.e., the sub-gradient equations of the objective function, then for all $s \in \mathbb{N}$, we always have 
        \begin{equation}\label{eq:first-order}
            \left\{
            \begin{aligned}
                & 0 \in \partial_{\textbf{G}_{k,i}} f\left(\mathbf{G}_{<k,i}^{(s+1)},\mathbf{G}_{k,i}^{(s+1)},\mathbf{G}_{>k,i}^{(s)},\mathcal{X}^{(s)}\right) + \rho (\textbf{G}_{k,i}^{(s+1)} - \textbf{G}_{k,i}^{(s)}),\\
               & 0 \in \partial_{\mathcal{X}} f(\textbf{G}^{(s+1)}, \mathcal{X}^{(s+1)}) + \rho (\mathcal{X}^{(s+1)} - \mathcal{X}^{(s)}).
            \end{aligned}
            \right.
        \end{equation}

        Based on the sub-differentiability property, we have
        \begin{equation}
            \partial F(\mathcal{M}^{(s+1)}) = \left(\partial_{\textbf{G}} F(\textbf{G}^{(s+1)}, \mathcal{X}^{(s+1)}), \partial_{\mathcal{X}} F(\textbf{G}^{(s+1)}, \mathcal{X}^{(s+1)})\right),
        \end{equation}
        where
        \begin{equation}
            \left\{
            \begin{aligned}
                & \partial_{\textbf{G}} F(\textbf{G}^{(s+1)}, \mathcal{X}^{(s+1)}) = \partial_{\textbf{G}} f(\textbf{G}^{(s+1)}, \mathcal{X}^{(s+1)}),\\
                & \partial_{\mathcal{X}} F(\textbf{G}^{(s+1)}, \mathcal{X}^{(s+1)}) = \partial_{\mathcal{X}} f(\textbf{G}^{(s+1)}, \mathcal{X}^{(s+1)}).
            \end{aligned}
            \right.
        \end{equation}
        Thus, we have the triangle inequality as follows,
        \begin{equation}
            \begin{aligned}
                \left\|\partial F(\mathcal{M}^{(s+1)})\right\|_{F}  \leq \sum_{k=1}^{N} \sum_{i \neq k}^{N} \left\|\partial_{\textbf{G}_{k,i}} f(\textbf{G}^{(s+1)}, \mathcal{X}^{(s+1)}) \right\|_{F} + \left\|\partial_{\mathcal{X}} f(\textbf{G}^{(s+1)}, \mathcal{X}^{(s+1)})\right\|_{F}
            \end{aligned}
        \end{equation}
        
        Substitute into the first-order optimal condition (\ref{eq:first-order}), leading to
        \begin{equation}
            \begin{aligned}
                & \left\|\partial F(\mathcal{M}^{(s+1)})\right\|_{F} \leq \sum_{k=1}^{N} \sum_{i \neq k}^{N} \left\|\partial_{\textbf{G}_{k,i}} f(\textbf{G}^{(s+1)}, \mathcal{X}^{(s+1)}) - \rho (\textbf{G}_{k,i}^{(s+1)} - \textbf{G}_{k,i}^{(s)}) \right.\\
                & \left.- \partial_{\textbf{G}_{k,i}}f\left(\mathbf{G}_{<k,i}^{(s+1)},\mathbf{G}_{k,i}^{(s+1)},\mathbf{G}_{>k,i}^{(s)},\mathcal{X}^{(s)}\right)\right\|_{F}+\rho \left\|\mathcal{X}^{(s+1)} -  \mathcal{X}^{(s)}\right\|_{F}\\
                & \leq \sum_{k=1}^{N} \sum_{i \neq k}^{N} \left\|\partial_{\textbf{G}_{k,i}} f(\textbf{G}^{(s+1)}, \mathcal{X}^{(s+1)})  - \partial_{\textbf{G}_{k,i}}f\left(\mathbf{G}_{<k,i}^{(s+1)},\mathbf{G}_{k,i}^{(s+1)},\mathbf{G}_{>k,i}^{(s)},\mathcal{X}^{(s)}\right)\right\|_{F} \\
                & + \sum_{k=1}^{N} \sum_{i \neq k}^{N}  \rho \left\|\textbf{G}_{k,i}^{(s+1)} - \textbf{G}_{k,i}^{(s)} \right\|_{F} +\rho \left\|\mathcal{X}^{(s+1)} -  \mathcal{X}^{(s)}\right\|_{F}. 
            \end{aligned}
        \end{equation}
        
        Since $f \in C^{1}$ and $h \equiv 0$, the partial derivatives $\left\{\partial_{\textbf{G}_{k,i}}F(\mathcal{M})\right\}$ is $\left\{L_{k,i}\right\}_{k\in [N], i \in [N] \setminus k}$-Lipschitz continuous. Then, we easily have
        \begin{equation}
            \begin{aligned}
                & \left\|\partial_{\textbf{G}_{k,i}} f(\textbf{G}^{(s+1)}, \mathcal{X}^{(s+1)}) - \partial_{\textbf{G}_{k,i}}f\left(\mathbf{G}_{<k,i}^{(s+1)},\mathbf{G}_{k,i}^{(s+1)},\mathbf{G}_{>k,i}^{(s)},\mathcal{X}^{(s)}\right)\right\|_{F}\\
                & \leq L_{k,i} \left\|\left\{\textbf{G}_{>k,i}^{(s+1)} - \textbf{G}_{>k,i}^{(s)}\right\}, \mathcal{X}^{(s+1)} - \mathcal{X}^{(s)}\right\|_{F},
            \end{aligned}
        \end{equation}
        Thus, a backsubstitution yields
        \begin{equation}
            \begin{aligned}
                \left\|\partial F(\mathcal{M}^{(s+1)})\right\|_{F} & \leq \sum_{k=1}^{N} \sum_{i \neq k}^{N} L_{k,i} \left\|\left\{\textbf{G}_{>k,i}^{(s+1)} - \textbf{G}_{>k,i}^{(s)}\right\}, \mathcal{X}^{(s+1)} - \mathcal{X}^{(s)}\right\|_{F}\\
                & + \sum_{k=1}^{N} \sum_{i \neq k}^{N}  \rho \left\|\textbf{G}_{k,i}^{(s+1)} - \textbf{G}_{k,i}^{(s)} \right\|_{F} +\rho \left\|\mathcal{X}^{(s+1)} -  \mathcal{X}^{(s)}\right\|_{F}\\
                & \leq \left(\sum_{k=1}^{N} \sum_{i \neq k}^{N} L_{k,i} + N(N-1) \right) \left\|\mathcal{M}^{(s+1)} - \mathcal{M}^{(s)}\right\|_{F} \\
                & = L_{F} \left\| \mathcal{M}^{(s+1)} - \mathcal{M}^{(s)}\right\|_{F},
            \end{aligned}
        \end{equation}
        where $L_{F} = N(N-1)\rho + \sum_{k=1}^{N} \sum_{i \neq k}^{N} L_{k,i}$. Therefore, there exists $\mathcal{P}^{(s+1)} \in \partial F(\mathcal{M}^{(s+1)})$ such that 
        \begin{equation}
            \left\|\mathcal{P}^{(s+1)}\right\|_{F} \leq L_{F} \left\| \mathcal{M}^{(s+1)} - \mathcal{M}^{(s)}\right\|_{F}.
        \end{equation}
        The relative error condition is proved.
    \end{proof}

\clearpage
{
\bibliographystyle{plain}
\bibliography{main}
}

\end{document}